\documentclass[preprint,11pt]{article}

\RequirePackage[colorlinks,citecolor=blue,urlcolor=blue]{hyperref}
\usepackage{amssymb}
\usepackage{mathrsfs}
\usepackage{graphicx}
\usepackage{colortbl,dcolumn}
\usepackage{tabularx}
\usepackage{amsmath}
\usepackage{psfrag}
\usepackage{booktabs}
\usepackage{cite}
\usepackage{amsthm}
\usepackage{float}
\usepackage{subfigure}

\numberwithin{equation}{section}
\usepackage[top=1.2in, bottom=1.2in, left=1in, right=1in]{geometry}

\newcommand{\C}{\mathbb{C}}
\newcommand{\R}{\mathbb{R}}
\newcommand{\N}{\mathbb{N}}
\newcommand{\E}{\mathbb{E}}
\renewcommand{\P}{\mathbb{P}}

\newcommand{\bi}{\mathbf{i}}
\newtheorem{tm}{Theorem}[section]

\newtheorem{prop}{Proposition}[section]

\newtheorem{rk}{Remark}

\allowdisplaybreaks \allowdisplaybreaks[4]

\begin{document}

\title{\bf Convergence in probability of an ergodic and conformal multi-symplectic numerical scheme for a damped stochastic NLS equation}

       \author{
        { Jialin Hong\footnotemark[1], Lihai Ji\footnotemark[2], and Xu Wang\footnotemark[3]}\\
      }
       \maketitle

        \footnotetext{\footnotemark[1]\footnotemark[3]Institute of Computational Mathematics and Scientific/Engineering Computing, Academy of Mathematics and Systems Science, Chinese Academy of Sciences, Beijing 100190, P.R.China. J. Hong and X. Wang are supported by National Natural Science Foundation of China (NO. 91530118, NO. 91130003, NO. 11021101 and NO. 11290142).}
     \footnotetext{\footnotemark[2]Institute of Applied Physics and Computational Mathematics, Beijing 100094, P.R.China. L. Ji is supported by National Natural Science Foundation of China (NO. 11471310, NO. 11601032).}
        \footnotetext{\footnotemark[3]Corresponding author: wangxu@lsec.cc.ac.cn.}

       \begin{abstract}
In this paper, we investigate the convergence order in probability of a novel ergodic numerical scheme for damped stochastic nonlinear Schr\"{o}dinger equation with an additive noise. Theoretical analysis shows that our scheme is of order one in probability under appropriate assumptions for the
initial value and noise. Meanwhile, we show that our scheme possesses the unique ergodicity and preserves the discrete conformal multi-symplectic conservation law. Numerical experiments are given to show the longtime behavior of the discrete charge and the time average of the numerical solution, and to test the convergence order, which verify our theoretical results. \\



\end{abstract}


\section{Introduction}
We consider the following weakly damped stochastic nonlinear Schr\"odinger (NLS) equation with an additive noise (see also \cite{CHW16,DO05})
\begin{equation}\label{model}
\left\{
\begin{aligned}
&d\psi-\bi (\Delta\psi+\bi\alpha \psi+\lambda|\psi|^2\psi)dt=\epsilon QdW,\quad t\ge0,\;x\in[0,1]\subset\R,\\
&\psi(t,0)=\psi(t,1)=0,\\
&\psi(0,x)=\psi_0(x)
\end{aligned}
\right.
\end{equation}
with a complex-valued Wiener process $W$ defined on a filtered probability space $(\Omega,\mathcal{F},\{\mathcal{F}_t\}_{t\ge0},\P)$, a linear positive operator $Q$ on $L^2:=L^2(0,1)$ and $\lambda=\pm1$. $\alpha>0$ is the absorption coefficient, $\epsilon\ge0$ describes the size of the noise. In addition, $Q$ is assumed to commute with $\Delta$ and satisfies $Qe_k=\sqrt{\eta_k}e_k$, which $\{e_k\}_{k\ge1}$ is an eigenbasis of $\Delta$ with Dirichlet boundary condition in $L_0^2:=L^2_0(0,1)$. Throughout the paper, the subscript $0$ represents the homogenous boundary condition. Some additional assumptions on $Q$ will be given bellow. The Karhunen--Lo\`eve expansion yields
\begin{align*}
QW(t,x)=\sum_{k=0}^{\infty}\sqrt{\eta_k} e_k(x)\beta_k(t),\quad t\geq0,\quad x\in[0,1],
\end{align*}
where $\beta_k=\beta_k^1+\bi\beta_k^2$ with $\beta_k^1$ and $\beta_k^2$ being its real and imaginary parts, respectively. In addition, we assume that $\{\beta_k^i\}_{k\ge1,i=1,2}$ is a family of $\R$-valued independent identified Brownian motions.

As is well known, for the deterministic cubic NLS equation, the charge of the solution is a constant, i.e., $\|\psi(t)\|_{L^{2}}=\|\psi_{0}\|_{L^{2}}$. However, for \eqref{model} with a damped term and an additive noise in addition, the charge is no longer preserved and satisfies (see Proposition 2.1 in Section 2)
\begin{align}\label{ergodic}
\E\|\psi(t)\|_{L^2}^2=e^{-2\alpha t}\E\|\psi_0\|_{L^2}^2+\frac{\epsilon^{2}\eta}{\alpha}(1-e^{-2\alpha t}),
\end{align}
where $\eta:=\sum_{k=1}^{\infty}\eta_k<\infty$. From this equation, it can be seen that the damped term is necessary to ensure the uniform boundedness of the solution in stochastic case, which also ensures the existence of invariant measures. Indeed, if $\alpha=0$ and $\epsilon\neq0$, the $L^2$ norm grows linearly in time. Moreover, if there exists a unique invariant measure $\mu$ which satisfies
\begin{align}
\lim_{T\to\infty}\frac1T\int_0^T\E f(\psi(s))ds=\int_{H_0^1}fd\mu,\quad\forall~ f\in C_b(H_0^1)
\end{align}
with $H^1_0:=H^1_0(0,1)$, we say that the random process $\psi(t)$ is uniquely ergodic \cite{DO05}. The interested readers are referred to \cite{daprato,DO05} and references therein for the  study of ergodicity with respect to the exact solution of  stochastic PDEs. We also refer to \cite{abdulle,mattingly,talay,T02} for the study of ergodicity as well as approximate error with respect to numerical solutions of stochastic ODEs and to \cite{brehier,brehier2,brehier3} for those of parabolic stochastic PDEs. For damped stochastic NLS equation \eqref{model}, the authors in \cite{DO05} prove both the existence and the uniqueness of the invariant measure, and \cite{CHW16} presents an ergodic fully discrete scheme which is of order $2$ in space for $\lambda=0$ or $\pm1$ and order $\frac12$ in time for $\lambda=0$ or $-1$ in the weak sense.

The main goal of this work is to construct a fully discrete scheme of \eqref{model} which could inherit both the unique ergodicity and some other internal properties of the original equation, e.g., conformal multi-symplectic property (see \cite{MNS13} for a detailed description in deterministic case), and to give the optimal convergence order of the proposed scheme in probability. To this end, we first apply the central finite difference scheme to \eqref{model} in spatial direction to get a semi-discretized equation, whose solution is shown to be symplectic and uniformly bounded. Then, a splitting technique is used to discretize the semi-discretized equation and obtain an explicit fully discrete scheme. We show that the proposed scheme possesses a conformal multi-symplectic conservation law with its solution uniformly bounded. Thanks to the non-degeneracy of the additive noise, the numerical solution is also shown to be irreducible and strong Feller, which yields the uniqueness of the invariant measure.
Due to the fact that the nonlinear term of \eqref{model} is not global Lipschitz, it is particularly challenging and difficulty to analyze the convergence order of the proposed scheme. Motivated by \cite{BD06,L13}, we construct a truncated equation with a global Lipschitz nonlinear term such that the proposed scheme applied to the truncated equation shows order one in mean-square sense under appropriate hypothesis on initial value and noise. We then construct a submartingale based on which we finally derive convergence order one in probability for the original equation in temporal direction. To the best of our knowledge, there has been no work in the literature which constructs schemes with both ergodicity and conformal multi-symplecticity to \eqref{model}.

The rest of the paper is organized as follows. We show the conformal multi-symplecticity and the charge evolution for \eqref{model} in section 2. In section 3, we construct a fully discrete scheme, which could inherit both the ergodicity and the conformal multi-symplecitcity of the original system. In section 4, we introduce a truncated equation, based on which we derive convergence order one in probability for the proposed scheme. Numerical experiments are carried out in section 5 to verify our theoretical results.

\section{Damped stochastic NLS equation}
This section is devoted to investigate the internal properties of \eqref{model}. We define the space-time white noise $\dot{\chi}=\frac{dW}{dt}$, set $\psi=p+\bi q$, $\dot{\chi}=\dot{\chi}_1+\bi\dot{\chi}_2$ with $p$, $q$, $\dot{\chi}_1=\frac{dW_1}{dt}$ and $\dot{\chi}_2=\frac{dW_2}{dt}$ being real-valued functions, and rewrite \eqref{model} as
\begin{equation}\label{pq}
\left\{
\begin{aligned}
\begin{split}
p_t+q_{xx}+\alpha p+\lambda(p^2+q^2)q&=\epsilon Q\dot{\chi}_1,\\[2mm]
-q_t+p_{xx}-\alpha q+\lambda(p^2+q^2)p&=-\epsilon Q\dot{\chi}_2.
\end{split}
\end{aligned}
\right.
\end{equation}
Denoting $v=p_x$, $w=q_x$, $z=(p,q,v,w)^T$, above equations can be transformed into a compact form
\begin{align}\label{multisym}
Md_tz+K\partial_xzdt=-\alpha Mzdt+\nabla S_0(z)dt+\nabla S_1(z)\circ dW_1+\nabla S_2(z)\circ dW_2,
\end{align}
where
\begin{equation*}M=
\left(
\begin{array}{cccc}
0&-1&0&0\\
1&0&0&0\\
0& 0&0 &0\\
0&0&0&0\\
\end{array}
\right),\quad K=
\left(
\begin{array}{cccc}
0&0&1&0\\
0&0&0&1\\
-1& 0&0 &0\\
0&-1&0&0\\
\end{array}
\right)
\end{equation*}
and
\begin{equation*}
S_0(z)=-\frac{\lambda}4(p^2+q^2)^2-\frac12(v^2+w^2),\quad
S_1(z)=\epsilon Qq,\quad S_2(z)=-\epsilon Qp.
\end{equation*}

In the sequel, we use the notations $L^2:=L^2(0,1)$, $H^p:=H^p(0,1)$ and denote the domain of operators $\Delta^{\frac{p}2}$ with Dirichlet boundary condition by
\begin{align*}
\dot{H}^p:=D(\Delta^{\frac{p}2})=\left\{u\in L^2_0\bigg{|}\|u\|_{\dot{H}^p}:=\|\Delta^{\frac{p}2}u\|_{L^2}=\sum_{k=1}^{\infty}(k\pi)^p|(u,e_k)|^2\le\infty\right\},\quad p\ge1
\end{align*}
with $(u,v):=\int_0^1u(x)\overline{v}(x)dx$ for all $u,v\in L^2$ and $e_k(x)=\sqrt{2}\sin(k\pi x)$. Furthermore,  we denote the set of Hilbert-Schimdt operators from $L^2$ to $\dot{H}^p$ by $\mathcal{L}_2^p$ with norm $$\|Q\|_{\mathcal{L}_2^p}:=\sum_{k=1}^{\infty}\|Qe_k\|^2_{\dot{H}^p},\quad p\ge1.$$
Without pointing out, the equations below hold in the sense $\P$-a.s. We now prove that \eqref{model} possesses the stochastic conformal multi-symplectic structure, whose definition is also given in the following theorem.
\begin{tm}
Eq. \eqref{model} is a stochastic conformal multi-symplectic Hamiltionian system, and preserves the stochastic conformal multi-symplectic conservation law
\begin{align*}
d_t\omega(t,x)+\partial_x\kappa(t,x)dt=-\alpha\omega(t,x)dt,
\end{align*}
which means
\begin{equation}\label{2forms}
\begin{split}
\int_{x_0}^{x_1}\omega(t_1,x)dx
-\int_{x_0}^{x_1}\omega(t_0,x)dx
&+\int_{t_0}^{t_1}\kappa(t,x_1)dt
-\int_{t_0}^{t_1}\kappa(t,x_0)dt\\
&=-\int_{x_0}^{x_1}\int_{t_0}^{t_1}\alpha\omega(t,x)dtdx,
\end{split}
\end{equation}
where $\omega=\frac12dz\wedge Mdz$ and $\kappa=\frac12dz\wedge Kdz$ are two differential 2-forms associated with two skew-symmetric matrices $M$ and $K$.
\end{tm}
\begin{proof}
To simplify the proof, we denote $(z_1,z_2,z_3,z_4):=(p,q,v,w)=z^T$ and $(z_l)_t^x:=z_l(t,x)$ for $l=1,2,3,4.$
Noticing that
$\omega=dz_2\wedge dz_1$
and
$\kappa=dz_1\wedge dz_3+dz_2\wedge dz_4$,
thus we have
\begin{equation}\label{H}
\begin{split}
&\int_{x_0}^{x_1}\omega(t_1,x)dx
-\int_{x_0}^{x_1}\omega(t_0,x)dx\\
=&\int_{x_0}^{x_1}\Big{[}d(z_2)_{t_1}^x\wedge d(z_1)_{t_1}^x-d(z_2)_{t_0}^x\wedge d(z_1)_{t_0}^x\Big{]}dx\\
=&\int_{x_0}^{x_1}\Bigg{[}\left(\sum_{l=1}^4\frac{\partial(z_2)_{t_1}^x}{\partial(z_l)_{t_0}^{x_0}}d(z_l)_{t_0}^{x_0}\right)\wedge\left(\sum_{i=1}^4\frac{\partial(z_1)_{t_1}^x}{\partial(z_i)_{t_0}^{x_0}}d(z_i)_{t_0}^{x_0}\right)\\
&-\left(\sum_{l=1}^4\frac{\partial(z_2)_{t_0}^x}{\partial(z_l)_{t_0}^{x_0}}d(z_l)_{t_0}^{x_0}\right)\wedge\left(\sum_{i=1}^4\frac{\partial(z_1)_{t_0}^x}{\partial(z_i)_{t_0}^{x_0}}d(z_i)_{t_0}^{x_0}\right)\Bigg{]}dx\\
=&\sum_{l=1}^4\sum_{i=1}^4\left[\int_{x_0}^{x_1}\left(\frac{\partial(z_2)_{t_1}^x}{\partial(z_l)_{t_0}^{x_0}}\frac{\partial(z_1)_{t_1}^x}{\partial(z_i)_{t_0}^{x_0}}
-\frac{\partial(z_2)_{t_0}^x}{\partial(z_l)_{t_0}^{x_0}}\frac{\partial(z_1)_{t_0}^x}{\partial(z_i)_{t_0}^{x_0}}\right)dx\right]
d(z_l)_{t_0}^{x_0}\wedge d(z_i)_{t_0}^{x_0}\\
=&:\sum_{l=1}^4\sum_{i=1}^4\mathcal{H}_{l,i}(t_1,x_1)
d(z_l)_{t_0}^{x_0}\wedge d(z_i)_{t_0}^{x_0},
\end{split}
\end{equation}
where $\mathcal{H}_{l,i}(t_1,x_1)=\int_{x_0}^{x_1}\left(\frac{\partial(z_2)_{t_1}^x}{\partial(z_l)_{t_0}^{x_0}}\frac{\partial(z_1)_{t_1}^x}{\partial(z_i)_{t_0}^{x_0}}
-\frac{\partial(z_2)_{t_0}^x}{\partial(z_l)_{t_0}^{x_0}}\frac{\partial(z_1)_{t_0}^x}{\partial(z_i)_{t_0}^{x_0}}\right)dx$.
Similarly, we obtain
\begin{equation}\label{M}
\begin{split}
&\int_{t_0}^{t_1}\kappa(t,x_1)dt
-\int_{t_0}^{t_1}\kappa(t,x_0)dt\\
=&\sum_{l=1}^4\sum_{i=1}^4\Bigg{[}\int_{t_0}^{t_1}\Bigg{(}
-\frac{\partial(z_1)_{t}^{x_1}}{\partial(z_i)_{t_0}^{x_0}}\frac{\partial(z_3)_{t}^{x_1}}{\partial(z_l)_{t_0}^{x_0}}
+\frac{\partial(z_1)_{t}^{x_0}}{\partial(z_i)_{t_0}^{x_0}}\frac{\partial(z_3)_{t}^{x_0}}{\partial(z_l)_{t_0}^{x_0}}\\
&-\frac{\partial(z_2)_{t}^{x_1}}{\partial(z_l)_{t_0}^{x_0}}\frac{\partial(z_4)_{t}^{x_1}}{\partial(z_i)_{t_0}^{x_0}}
+\frac{\partial(z_2)_{t}^{x_0}}{\partial(z_l)_{t_0}^{x_0}}\frac{\partial(z_4)_{t}^{x_0}}{\partial(z_i)_{t_0}^{x_0}}\Bigg{)}dt\Bigg{]}
d(z_l)_{t_0}^{x_0}\wedge d(z_i)_{t_0}^{x_0}\\
=&:\sum_{l=1}^4\sum_{i=1}^4\mathcal{M}_{l,i}(t_1,x_1)d(z_l)_{t_0}^{x_0}\wedge d(z_i)_{t_0}^{x_0}
\end{split}
\end{equation}
and
\begin{align}
\int_{x_0}^{x_1}\int_{t_0}^{t_1}\alpha\omega(t,x)dtdx
=&2\alpha\sum_{l=1}^4\sum_{i=1}^4\left[\int_{x_0}^{x_1}\int_{t_0}^{t_1}\left(\frac{\partial(z_2)_{t}^x}{\partial(z_l)_{t_0}^{x_0}}\frac{\partial(z_1)_{t}^x}{\partial(z_i)_{t_0}^{x_0}}\right)dtdx\right]
d(z_l)_{t_0}^{x_0}\wedge d(z_i)_{t_0}^{x_0}\nonumber\\\label{N}
=&:2\alpha\sum_{l=1}^4\sum_{i=1}^4\mathcal{N}_{l,i}(t_1,x_1)d(z_l)_{t_0}^{x_0}\wedge d(z_i)_{t_0}^{x_0}.
\end{align}
Adding \eqref{H}, \eqref{M} and \eqref{N} together, we can find out that equation \eqref{2forms} holds if
\begin{align}\label{HMN}
\mathcal{H}_{l,i}(t_1,x_1)+\mathcal{M}_{l,i}(t_1,x_1)+2\alpha\mathcal{N}_{l,i}(t_1,x_1)=0
\end{align}
for any $l,i=1,2,3,4$, $t_1\in\R_+$ and $x_1\in\R$.
In fact, rewritting \eqref{pq} as
\begin{equation*}
\left\{
\begin{aligned}
d_{t}z_1=&-\partial_x(z_4)dt-\alpha z_1dt+\frac{\partial S_0(z)}{\partial z_2}dt+\epsilon QdW_1,\\
d_{t}z_2=&\partial_x(z_3)dt-\alpha z_2dt-\frac{\partial S_0(z)}{\partial z_1}dt+\epsilon QdW_2
\end{aligned}
\right.
\end{equation*}
and taking partial derivatives with respect to $(z_i)_{t_0}^{x_0}$ and $(z_l)_{t_0}^{x_0}$ respectively, we have
\begin{equation*}
\left\{
\begin{aligned}
d_{t}\frac{\partial(z_1)_{t}^x}{\partial(z_i)_{t_0}^{x_0}}=&-\frac{\partial}{\partial x}\frac{\partial(z_4)_{t}^x}{\partial(z_i)_{t_0}^{x_0}}dt-\alpha \frac{\partial(z_1)_{t}^x}{\partial(z_i)_{t_0}^{x_0}}dt+\sum_{l=1}^4\frac{\partial S_1(z)}{\partial (z_2)_{t}^x\partial(z_l)_{t}^{x}}\frac{\partial(z_l)_t^x}{\partial(z_i)_{t_0}^{x_0}}dt,\\
d_{t}\frac{\partial(z_2)_{t}^x}{\partial(z_l)_{t_0}^{x_0}}=&\frac{\partial}{\partial x}\frac{\partial(z_3)_{t}^x}{\partial(z_l)_{t_0}^{x_0}}dt-\alpha \frac{\partial(z_2)_{t}^x}{\partial(z_l)_{t_0}^{x_0}}dt-\sum_{i=1}^4\frac{\partial S_1(z)}{\partial (z_1)_t^x\partial(z_i)_{t}^{x}}\frac{\partial(z_i)_{t}^{x}}{\partial(z_l)_{t_0}^{x_0}}dt.
\end{aligned}
\right.
\end{equation*}
Furthermore,
\begin{align*}
d_{t_1}\mathcal{H}_{l,i}(t_1,x_1)=&\int_{x_0}^{x_1}\frac{\partial(z_1)_{t_1}^x}{\partial(z_i)_{t_0}^{x_0}}d_{t_1}\left(\frac{\partial(z_2)_{t_1}^x}{\partial(z_l)_{t_0}^{x_0}}\right)+\frac{\partial(z_2)_{t_1}^x}{\partial(z_l)_{t_0}^{x_0}}d_{t_1}\left(\frac{\partial(z_1)_{t_1}^x}{\partial(z_i)_{t_0}^{x_0}}\right)dx\\
=&\int_{x_0}^{x_1}\frac{\partial(z_1)_{t_1}^x}{\partial(z_i)_{t_0}^{x_0}}\left(\frac{\partial}{\partial x}\frac{\partial(z_3)_{t_1}^x}{\partial(z_l)_{t_0}^{x_0}}-\alpha \frac{\partial(z_2)_{t_1}^x}{\partial(z_l)_{t_0}^{x_0}}\right)dt_1dx\\
&-\int_{x_0}^{x_1}\frac{\partial(z_2)_{t_1}^x}{\partial(z_l)_{t_0}^{x_0}}\left(\frac{\partial}{\partial x}\frac{\partial(z_4)_{t_1}^x}{\partial(z_i)_{t_0}^{x_0}}+\alpha \frac{\partial(z_1)_{t_1}^x}{\partial(z_i)_{t_0}^{x_0}}\right)dt_1dx\\
=&\frac{\partial(z_1)_{t_1}^{x_1}}{\partial(z_i)_{t_0}^{x_0}}\frac{\partial(z_3)_{t_1}^{x_1}}{\partial(z_l)_{t_0}^{x_0}}
-\frac{\partial(z_1)_{t_1}^{x_0}}{\partial(z_i)_{t_0}^{x_0}}\frac{\partial(z_3)_{t_1}^{x_0}}{\partial(z_l)_{t_0}^{x_0}}
-\frac{\partial(z_2)_{t_1}^{x_1}}{\partial(z_l)_{t_0}^{x_0}}\frac{\partial(z_4)_{t_1}^{x_1}}{\partial(z_i)_{t_0}^{x_0}}\\
&+\frac{\partial(z_2)_{t_1}^{x_0}}{\partial(z_l)_{t_0}^{x_0}}\frac{\partial(z_4)_{t_1}^{x_0}}{\partial(z_i)_{t_0}^{x_0}}
-2\alpha\int_{x_0}^{x_1}\frac{\partial(z_1)_{t_1}^x}{\partial(z_i)_{t_0}^{x_0}}\frac{\partial(z_2)_{t_1}^x}{\partial(z_l)_{t_0}^{x_0}}dxdt_1\\
=&-d_{t_1}\mathcal{M}_{l,i}(t_1,x_1)-2\alpha d_{t_1}\mathcal{N}_{l,i}(t_1,x_1),
\end{align*}
which together with the fact that $\mathcal{H}_{l,i}(t_0,x_1)+\mathcal{M}_{l,i}(t_0,x_1)+2\alpha\mathcal{N}_{l,i}(t_0,x_1)=0$
yields \eqref{HMN}.
We hence complete the proof.
\end{proof}
Next, we show that the charge of the solution $\psi(t)$, although it is not conserved anymore, satisfies an exponential type evolution law. 
\begin{prop}\label{exactmoment}
Assume that $\E\|\psi_0\|_{L^2}^2<\infty$, then the solution of \eqref{model} is uniformly bounded with
\begin{align}\label{bounded}
\E\|\psi(t)\|_{L^2}^2=e^{-2\alpha t}\E\|\psi_0\|_{L^2}^2+\frac{\epsilon^{2}\eta}{\alpha}(1-e^{-2\alpha t}).
\end{align}
\end{prop}
\begin{proof}
The It\^o's formula applied to $\|\psi(t)\|_{L^2}$ yields
\begin{align*}
d\|\psi(t)\|^2_{L^2}=-2\alpha\|\psi(t)\|^2_{L^2}dt+2\epsilon\Re\left[\int_0^1\overline{\psi}QdxdW\right]+2\epsilon^{2}\eta dt,
\end{align*}
where $\Re[\cdot]$ denotes the real part of a complex value.
Taking expectation on both sides of above equation and solving the ordinary differential equation, we derive
\begin{equation*}
\begin{split}
\E\|\psi(t)\|^2_{L^2}
=&e^{-2\alpha t}\left(\int_0^t2\epsilon^{2}\eta e^{2\alpha s}ds+\E\|\psi_0\|^2_{L^2}\right)\\
&=e^{-2\alpha t}\E\|\psi_0\|_{L^2}^2+\frac{\epsilon^{2}\eta}{\alpha}(1-e^{-2\alpha t}).
\end{split}
\end{equation*}
This concludes the proof.
\end{proof}
We hence get the conclusion that, the damped stochastic NLS equation \eqref{model} possesses the
stochastic conformal multi-symplectic conservation law \eqref{2forms} with its solution being uniformly bounded \eqref{bounded}, as well as  the unique ergodicity \cite{DO05}. A natural question is how to construct  numerical
schemes which could inherit the properties of \eqref{model} as many as possible, such as, stochastic conformal multi-symplectic conservation law, uniform boundedness of the solution and the unique ergodicity. 

\section{Ergodic fully discrete scheme}
We now focus on the construction of schemes which could inherit the stochastic conformal multi-symplecticity and the unique ergodicity.
\subsection{Spatial semi-discretization}
We apply central finite difference scheme to \eqref{model} and obtain
\begin{equation}\label{fenliang}
\left\{
\begin{aligned}
&d\psi_j-\bi \left(\frac{\psi_{j+1}-2\psi_j+\psi_{j-1}}{h^2}+\bi\alpha\psi_j+\lambda|\psi_j|^2\psi_j\right)dt=\epsilon\sum_{k=1}^P\sqrt{\eta_k}e_k(x_j)d\beta_k(t),\\
&\psi_0(t)=\psi_{J+1}(t)=0,\\
&\psi_j(0)=\psi_0(x_j),
\end{aligned}
\right.
\end{equation}
where $h$ is the uniform spatial step and $\psi_j:=\psi_j(t)$ is an approximation of $\psi(x_j,t)$ with $x_j=jh$, $j=1,2,\cdots,J$ and $(J+1)h=1$.  With notations $\Psi=(\psi_1,\cdots,\psi_J)^T\in\C^J$,  $\beta=(\beta_1,\cdots,\beta_P)^T\in\C^P$, $F(\Psi)=
\text{diag}\{|\psi_1|^2,\cdots,|\psi_J|^2\}$, $\Lambda=\text{diag}\{\sqrt{\eta_1^{}},\cdots,\sqrt{\eta_P^{}}\}$,
\begin{equation*}
A=\left(
\begin{array}{cccc}
-2&1 & & \\
 1&-2&1 & \\
&\ddots &\ddots &\ddots \\
 & &1 &-2\\
\end{array}
\right)\quad\text{and}\quad
\sigma=
\left(
\begin{array}{ccc}
e_1^{}(x_1^{})&\cdots &e_P^{}(x_1^{})\\
\vdots& &\vdots\\
e_1^{}(x_J^{})&\cdots&e_P^{}(x_J^{})
\end{array}
\right),
\end{equation*}
we rewrite \eqref{fenliang} into a finite dimensional stochastic differential equation
\begin{equation}\label{space}
\left\{
\begin{aligned}
&d\Psi-\bi \left(\frac1{h^2}A\Psi+\bi\alpha\Psi+\lambda F(\Psi)\Psi\right)dt=\epsilon\sigma\Lambda d\beta,\\
&\Psi(0)=(\psi_0(x_1^{}),\cdots,\psi_0(x_J^{}))^T.
\end{aligned}
\right.
\end{equation}

In the sequel, we denote the 2-norm for vectors or matrices by $\|\cdot\|$, i.e., $\|v\|=\left(\sum_{j=1}^J|v_j|^2\right)^{1/2}$ for a vector $v=(v_1,\cdots,v_J)^T\in\C^J$ and $\|A\|$=`the square root of the maximum eigenvalues of $A^TA$' for a matrix A. The solution of \eqref{space} is uniformly bounded, which is stated in the following proposition.
\begin{prop}\label{semimoment}
Assume that $\E\|\psi_0\|_{L^2}^2<\infty$, then the solution $\Psi$ of \eqref{space} is uniformly bounded with
\begin{align}
h\E\|\Psi(t)\|^2\le e^{-2\alpha t}h\E\|\Psi(0)\|^2+\frac{2\epsilon^2\eta^{(P)}}{\alpha}(1-e^{-2\alpha t}),
\end{align}
where $\eta^{(P)}:=\sum_{k=1}^P\eta_k^{}$.
\end{prop}
\begin{proof}
Similar to the proof of Proposition \ref{exactmoment}, we apply It\^o's formula to $\|\Psi(t)\|^2$ and obtain
\begin{equation}\label{semiito}
\begin{split}
d\|\Psi(t)\|^2&=2\Re[\overline{\Psi}^Td\Psi]+(\epsilon\sigma\Lambda d\overline{\beta})^T(\epsilon\sigma\Lambda d\beta)\\[1mm]
&=-2\alpha\|\Psi(t)\|^2dt+2\Re[\overline{\Psi}^T\epsilon\sigma\Lambda d\beta]\\
&+\epsilon^2\sum_{j=1}^J\left[\left(\sum_{k=1}^P\sqrt{\eta_k}e_k(x_j)d\overline{\beta_k}\right)^T\left(\sum_{k=1}^P\sqrt{\eta_k}e_k(x_j)d{\beta}_k\right)\right].
\end{split}
\end{equation}
Taking expectation on both sides of above equation leads to
\begin{align*}
d\E\|\Psi(t)\|^2=&-2\alpha \E\|\Psi(t)\|^2dt+2\epsilon^2\sum_{j=1}^J\sum_{k=1}^P\eta_ke_k^2(x_j)dt.
\end{align*}
Thus, multiplying above equation by $he^{2\alpha t}$ and taking integral from $0$ to $t$ leads to
\begin{align*}
\int_0^the^{2\alpha t}d\E\|\Psi(t)\|^2+\int_0^t2\alpha he^{2\alpha t}\E\|\Psi(t)\|^2dt
=2\epsilon^2h\sum_{j=1}^J\sum_{k=1}^P\eta_ke_k^2(x_j)\int_0^te^{2\alpha t}dt.
\end{align*}
Based on the fact that $\sum_{j=1}^Je^2_k(x_j)\le2J\le2h^{-1}$, we have
\begin{equation}\label{charge}
\begin{split}
e^{2\alpha t}h\E\|\Psi(t)\|^2-h\E\|\Psi(0)\|^2
=&\frac{\epsilon^2h}{\alpha}(e^{2\alpha t}-1)\sum_{j=1}^J\sum_{k=1}^P\eta_ke_k^2(x_j)\\
\le& \frac{2\epsilon^2\eta^{(P)}}{\alpha}(e^{2\alpha t}-1)
\end{split}
\end{equation}
which completes the proof. 

In addition, noticing that for $h=1/J$ and $x_j=jh,$ $j=1,\cdots,J$,
\begin{equation*}
\begin{split}
\E\|\psi_0\|_{L^2}^2=&\E\sum_{j=1}^J\int_{x_{j-1}}^{x_j}|\psi_0(x)|^2dx\\
&=\E\sum_{j=1}^J|\psi_0(x_j)|^2h+O(h)=h\E\|\Psi(0)\|^2+O(h),
\end{split}
\end{equation*}
we get the uniform boundedness under the assumption $\E\|\psi_0\|_{L^2}^2<\infty$.
\end{proof}

\begin{rk}\label{symeuler}
Scheme \eqref{fenliang} is equivalent to the symplectic Euler scheme applied to \eqref{multisym}, i.e.,
\begin{equation*}\label{space2}
\left\{
\begin{aligned}
&p_{j+1}-p_j=hv_{j+1},\\[4mm]
&q_{j+1}-q_j=hw_{j+1},\\
&v_{j+1}-v_j=h(q_j)_t+\alpha hq_j-h((p_j)^2+(q_j)^2)p_j-\epsilon\sum_{k=1}^P\sqrt{\eta_k}e_k(x_j)\frac{d\beta^2_k(t)}{dt},\\
&w_{j+1}-w_j=-h(p_j)_t-\alpha hp_j-h((p_j)^2+(q_j)^2)q_j+\epsilon\sum_{k=1}^P\sqrt{\eta_k}e_k(x_j)\frac{d\beta^1_k(t)}{dt}.
\end{aligned}
\right.
\end{equation*}
\end{rk}

\subsection{Full discretization}
To construct a fully discrete scheme, which could inherit the properties of \eqref{model}, we are motivated by splitting techniques. We drop the linear terms and stochastic term for the moment and consider the following equation
\begin{align}\label{nonlinear}
d\Psi(t)-\bi \lambda F(\Psi(t))\Psi(t)dt=0
\end{align}
first. Multiplying $\overline{F(\Psi(t))}$ to both sides of \eqref{nonlinear} and taking the imaginary part, we obtain $\|\Psi(t)\|^2=\|\Psi(0)\|^2$, which implies that $F(\Psi(t))=F(\Psi(0))$. Thus, \eqref{nonlinear} is shown to possess a unique solution $\Psi(t)=e^{\bi\lambda F(\Psi(0))t}\Psi(0)$. 

For linear equation
\begin{align*}
d\Psi(t)-\bi \left(\frac1{h^2}A\Psi(t)+\bi\alpha\Psi(t)\right)dt=\epsilon\sigma\Lambda d\beta,
\end{align*}
a modified mid-point scheme is applied to obtain its full discretization.
Now we can define the following splitting schemes initialized with $\Psi^0=\Psi(0)$,
\begin{align}\label{full1}
&\Psi^{n+1}=e^{-\frac12\alpha\tau}\tilde{\Psi}^n+\bi\frac{\tau}{h^2}A\frac{\Psi^{n+1}+e^{-\frac12\alpha\tau}\tilde{\Psi}^n}2-\frac12\alpha\tau\frac{\Psi^{n+1}+e^{-\frac12\alpha\tau}\tilde{\Psi}^n}2+\epsilon\sigma\Lambda\delta_{n+1}\beta,\\\label{full2}
&\tilde{\Psi}^n=e^{\bi\lambda F(\Psi^n)\tau}\Psi^n,
\end{align}
where $\Psi^n=(\psi_1^n,\cdots,\psi_J^n)^T\in\C^{J}$, $\tau$ denotes the uniform time step, $\delta_{n+1}\beta=\beta(t_{n+1})-\beta(t_n)$ and $t_n=n \tau$, $n\in\N$.
Noticing that schemes \eqref{full1}--\eqref{full2} can be rewritten as
\begin{equation}\label{full}
\begin{split}
\Psi^{n+1}-e^{f(\Psi^n)}\Psi^n=&\bi\frac{\tau}{2h^2}A\left(\Psi^{n+1}+e^{f(\Psi^n)}\Psi^n\right)\\
&-\frac14\alpha\tau\left(\Psi^{n+1}+e^{f(\Psi^n)}\Psi^n\right)+\epsilon\sigma\Lambda\delta_{n+1}\beta,
\end{split}
\end{equation}
which can also be expressed in the following explicit form
\begin{equation}\label{explicitscheme}
\begin{split}
\Psi^{n+1}=&\left(I-\frac{\bi\tau}{2h^2}A+\frac14\alpha\tau I\right)^{-1}\left(I+\frac{\bi\tau}{2h^2}A-\frac14\alpha\tau I\right)e^{f(\Psi^n)}\Psi^n\\
&+\left(I-\frac{\bi\tau}{2h^2}A+\frac14\alpha\tau I\right)^{-1}\epsilon\sigma\Lambda\delta_{n+1}\beta
\end{split}
\end{equation}
with $I$ denoting the identity matrix and $f(\Psi^n)=\left(-\frac12\alpha I+\bi\lambda F(\Psi^n)\right)\tau$. Thus, there uniquely exists a family of $\{\mathcal{F}_{t_n}\}_{n\ge1}$ adapted solutions $\{\Psi^n\}_{n\ge1}$ of \eqref{full} for sufficiently small $\tau$.

As for the proposed splitting schemes, \eqref{full2} coincides with the exact solution of the Hamiltonian system $d\Psi(t)-\bi \lambda F(\Psi(t))\Psi(t)dt=0$. It then suffices to show that \eqref{full1} possesses the conformal multi-symplectic conservation law, which is stated in the following theorem.
\begin{tm}
Scheme \eqref{full1} possesses the discrete conformal multi-symplectic conservation law
\begin{align*}
&e^{-\alpha\tau}\frac{dz_j^{n+1}\wedge Mdz_j^{n+1}-dz_j^{n}\wedge Mdz_j^{n}}{\tau}+\frac{dz_j^{n+\frac12}\wedge(K_1dz_{j+1}^{n+\frac12}-K_2dz_{j-1}^{n+\frac12})}{h}\\
=&-\frac12\alpha dz_j^{n+\frac12}\wedge Mdz_j^{n+\frac12}
\end{align*}
with $z_j^n=(p_j^n,q_j^n,v_j^n,w_j^n)^T$, $z_j^{n+\frac12}=\frac12(z_j^{n+1}+e^{-\frac12\alpha\tau}z_j^n)$,
\begin{equation*}K_1=
\left(
\begin{array}{cccc}
0&0&1&0\\
0&0&0&1\\
0& 0&0 &0\\
0&0&0&0\\
\end{array}
\right),\quad K_2=
\left(
\begin{array}{cccc}
0&0&0&0\\
0&0&0&0\\
-1& 0&0 &0\\
0&-1&0&0\\
\end{array}
\right)
\end{equation*}
and $K_1+K_2=K$.
\end{tm}
\begin{proof}
We denote $\tilde{\Psi}^n$ by $\Psi^n$ for convenience in this proof, and we have
\begin{align*}
\Psi^{n+1}=e^{-\frac12\alpha\tau}{\Psi}^n+\bi\frac{\tau}{h^2}A\frac{\Psi^{n+1}+e^{-\frac12\alpha\tau}{\Psi}^n}2-\frac12\alpha\tau\frac{\Psi^{n+1}+e^{-\frac12\alpha\tau}{\Psi}^n}2+\epsilon\sigma\Lambda\delta_{n+1}\beta
\end{align*}
with $\Psi^n=\left(\psi_1^n,\cdots,\psi_J^n\right)^T\in\C^J$.
Denote $\psi_j^n:=p_j^n+\bi q_j^n$ with its real part $p_j^n$ and imaginary part $q_j^n$, $\delta_{n+1}\beta=\delta_{n+1}\beta^1+\bi\delta_{n+1}\beta^2$, $v_{j+1}^{n}:=(p_{j+1}^{n}-p_j^n)h^{-1}$ and $w_{j+1}^{n}:=(q_{j+1}^{n}-q_j^n)h^{-1}$.
Noticing that the $j$-th component of $h^{-2}A\Psi^n$ can be expressed as $h^{-1}(v_{j+1}^n-v_j^n)+\bi h^{-1}(w_{j+1}^n-w_j^n)$, we decompose \eqref{full1} with its real and imaginary parts respectively and derive
\begin{equation*}\label{theoformula}
\left\{
\begin{aligned}
\frac{p_j^{n+1}-e^{-\frac12\alpha\tau}p_j^n}{\tau}+&\frac{w_{j+1}^{n+1}-w_j^{n+1}}{2h}+e^{-\frac12\alpha\tau}\frac{w_{j+1}^{n}-w_j^{n}}{2h}\\
=&-\frac14\alpha(p_j^{n+1}+e^{-\frac12\alpha\tau}p_j^n)+\epsilon\sigma\Lambda\delta_{n+1}\beta^1,\\
\frac{q_j^{n+1}-e^{-\frac12\alpha\tau}q_j^{n}}{\tau}-&\frac{v_{j+1}^{n+1}-v_j^{n+1}}{2h}-e^{-\frac12\alpha\tau}\frac{v_{j+1}^{n}-v_j^{n}}{2h}\\
=&-\frac14\alpha(q_j^{n+1}+e^{-\frac12\alpha\tau}q_j^n)+\epsilon\sigma\Lambda\delta_{n+1}\beta^2.
\end{aligned}
\right.
\end{equation*}
Combining formulae $v_{j+1}^{n}=(p_{j+1}^{n}-p_j^n)h^{-1}$, $w_{j+1}^{n}=(q_{j+1}^{n}-q_j^n)h^{-1}$ with above equations, we get
\begin{align*}
M\frac{z_j^{n+1}-e^{-\frac12\alpha\tau}z_j^n}{\tau}+K_1\frac{z_{j+1}^{n+\frac12}-z_j^{n+\frac12}}{h}
+K_2\frac{z_{j}^{n+\frac12}-z_{j-1}^{n+\frac12}}{h}
=-\frac12\alpha Mz_j^{n+\frac12}+\xi_j^{n+\frac12},
\end{align*}
where $\xi_j^{n+\frac12}:=(-\epsilon\sigma\Lambda\delta_{n+1}\beta^2,\epsilon\sigma\Lambda\delta_{n+1}\beta^1,-v_j^{n+\frac12},-w_j^{n+\frac12})^T$.
Taking differential in phase space on both sides of above equation, and performing wedge product with $dz_j^{n+\frac12}$ respectively, we show the discrete conformal multi-symplectic conservation law based on the symmetry of matrix $-K_1+K_2$ and the fact $dz_j^{n+\frac12}\wedge(-K_1+K_2)dz_j^{n+\frac12}=0$, $dz_j^{n+\frac12}\wedge d\xi_j^{n+\frac12}=0$.
\end{proof}
\begin{rk}
It is also feasible to show that schemes \eqref{full1}--\eqref{full2} are conformal symplectic in temporal direction, which together with Remark \ref{symeuler}, yields the conformal multi-symplecticity of the fully discrete scheme \eqref{full}.

\end{rk}

\begin{prop}\label{fullmoment}
Assume that $\E\|\psi_0\|_{L^2}^2<\infty$, $Q\in\mathcal{HS}(L^2,\dot{H}^2)$ and $P\le C_*(J+1)$ for some constant $C_*\ge1$, then the solution $\{\Psi^n\}_{n\ge1}$ of \eqref{full} is uniformly bounded, i.e.,
\begin{align}
h\E\|\Psi^n\|^2\le e^{-\alpha t_n}h\E\|\Psi^0\|^2+C
\end{align}
with $t_n=n\tau$ and constant $C$ depending on $\alpha,\epsilon,Q$ and $C_*$.
\end{prop}
\begin{proof}
We multiply $\overline{\left(\Psi^{n+1}+e^{f(\Psi^n)}\Psi^n\right)}^T$ to \eqref{full}, take the real part and expectation, and obtain
\begin{equation}\label{moment}
\begin{split}
&\E\|\Psi^{n+1}\|^2-e^{-\alpha\tau}\E\|\Psi^n\|^2\\
=&-\frac14\alpha\tau\E\|\Psi^{n+1}+e^{f(\Psi^n)}\Psi^n\|^2+\E\left[\Re\left[\overline{\left(\Psi^{n+1}-e^{f(\Psi^n)}\Psi^n\right)}^T\epsilon\sigma\Lambda\Delta_{n+1}\beta\right]\right]\\
=&-\frac14\alpha\tau\E\|\Psi^{n+1}+e^{f(\Psi^n)}\Psi^n\|^2
+\E\Bigg{[}\Re\Bigg{[}\bigg{(}-\bi\frac{\tau}{2h^2}A\overline{\left(\Psi^{n+1}+e^{f(\Psi^n)}\Psi^n\right)}\\
&-\frac14\alpha\tau\overline{\left(\Psi^{n+1}+e^{f(\Psi^n)}\Psi^n\right)}+\overline{\epsilon\sigma\Lambda\Delta_{n+1}\beta}\bigg{)}^T\epsilon\sigma\Lambda\Delta_{n+1}\beta\Bigg{]}\Bigg{]}\\
\le&-\frac14\alpha\tau\E\|\Psi^{n+1}+e^{f(\Psi^n)}\Psi^n\|^2+\frac18\alpha\tau\E\|\Psi^{n+1}+e^{f(\Psi^n)}\Psi^n\|^2\\
&+C\tau\E\|h^{-2}A\epsilon\sigma\Lambda\Delta_{n+1}\beta\|^2
+\frac18\alpha\tau\E\|\Psi^{n+1}+e^{f(\Psi^n)}\Psi^n\|^2\\
&+C\tau\E\|\epsilon\sigma\Lambda\Delta_{n+1}\beta\|^2+\E\|\epsilon\sigma\Lambda\Delta_{n+1}\beta\|^2.
\end{split}
\end{equation}
For the smooth functions $e_k(x)$, we have
\begin{align*}
\left|\Delta e_k(x_j)-\frac{e_k(x_{j+1})-2e_k(x_j)+e_k(x_{j-1})}{h^2}\right|\le Ck^4h^2\le Ck^2,\quad k\ge1
\end{align*}
based on the fact $kh\le P(J+1)^{-1}\le C_*$.
Thus,
\begin{equation}\label{sto1}
\begin{split}
&\E\|h^{-2}A\epsilon\sigma\Lambda\Delta_{n+1}\beta\|^2
=\epsilon^2\sum_{j=1}^J\E\left|\sum_{k=1}^P\sqrt{\eta_k^{}}\frac{e_k(x_{j+1})-2e_k(x_j)+e_k(x_{j-1})}{h^2}\Delta_{n+1}\beta_k\right|^2\\
&\le2\epsilon^2\sum_{j=1}^J\sum_{k=1}^P\eta_k^{}\left(|\Delta e_k(x_j)|+Ck^2\right)^2\tau
\le CJ\tau\sum_{k=1}^Pk^4\eta_k
\le Ch^{-1}\tau.
\end{split}
\end{equation}
In the last step, we have used the fact $\sum\limits_{k=1}^Pk^4\eta_k\le C\|Q\|_{\mathcal{L}_2^2}\le C$. Similarly,
\begin{align}\label{sto2}
\E\|\epsilon\sigma\Lambda\Delta_{n+1}\beta\|^2
=\epsilon^2\sum_{j=1}^J\E\left|\sum_{k=1}^P\sqrt{\eta_k^{}}e_k(x_j)\Delta_{n+1}\beta_k\right|^2
\le CJ\eta\tau
\le Ch^{-1}\tau.
\end{align}
Substituting \eqref{sto1} and \eqref{sto2} into \eqref{moment} and multiplying the result by $h$, we get
\begin{align*}
h\E\|\Psi^{n+1}\|^2\le e^{-\alpha\tau}h\E\|\Psi^n\|^2+C\tau
\le e^{-\alpha t_{n+1}}h\E\|\Psi^0\|^2+C\tau\frac{1-e^{-\alpha t_n}}{1-e^{-\alpha\tau}},
\end{align*}
which, together with the fact $\frac{1-e^{-\alpha t_n}}{1-e^{-\alpha\tau}}\le\frac1{1-(1-\alpha\tau)}=\frac1{\alpha\tau}$,
completes the proof.
\end{proof}

\begin{tm}
Under the assumptions in Proposition \ref{fullmoment} and $\eta_k>0$ for $k=1,\cdots,P$. The solution $\{\Psi^n\}_{n\ge1}$ of \eqref{full} is uniquely ergodic with a unique invariant measure, denoted by $\mu^{\tau}_h$, satisfying
\begin{align}
\lim_{N\to\infty}\frac1N\sum_{n=0}^{N-1}\E f(\Psi^n)=\int_{\C^J}fd\mu_h^{\tau},\quad\forall\;f\in C_b(\C^J).
\end{align}
\end{tm}
\begin{proof}
To show the existence of the invariant measures, we'll use a useful tool called Lyapunov function. For any fixed $h>0,$ we choose $V(\cdot):=h\|\cdot\|^2$ as the Lyapunov function, which satisfies that the level sets $K_c:=\{u\in\C^J:V(u)\le c\}$ are compact for any $c>0$ and $\E[V(\Psi^n)]\le V(\Psi^0)+C$ for any $n\in\N$. Thus, the Markov chain $\{\Psi^n\}_{n\in\N}$ possesses an invariant measure (see Proposition 7.10 in \cite{daprato}).

Now we show that $\{\Psi^n\}_{n\in\N}$ is irreducible and strong Feller (also known as the minorization condition in Assumption 2.1 of \cite{mattingly}), which yields the uniqueness of the invariant measure. In fact, for any $u,v\in\C^J$, we can derive from \eqref{full} that $\delta_1\beta$ can be chosen as
\begin{align*}
\epsilon\sigma\Lambda\delta_{1}\beta=v-e^{f(u)}u-\bi\frac{\tau}{2h^2}A\left(v+e^{f(u)}u\right)+\frac14\alpha\tau\left(v+e^{f(u)}u\right)
\end{align*}
such that $\Psi^0=u,\Psi^1=v$, where we have used the fact that $\sigma$ is full rank and $\Lambda$ is invertible.
Thus, we can conclude based on the homogenous property of the Markov chain  $\{\Psi^n\}_{n\in\N}$ that the transition kernel
$P_n(u,A):=\P(\Psi^n\in A|\Psi^0=u)>0$, which implies the irreducibility of the chain.
On the other hand, as $\delta_1\beta$ has $C^{\infty}$ density, it follows from \eqref{explicitscheme} that $\Psi^1$ also has $C^{\infty}$ density for any deterministic initial value $\Psi^0=\Psi(0)$. Then explicit construction shows that $\{\Psi^n\}_{n\in\N}$ possesses a family of $C^{\infty}$ density and is strong Feller.
\end{proof}

The theorems above are evidently consistent with the continuous results \eqref{ergodic}, \eqref{2forms} and \eqref{bounded}, respectively. The next result concerns the error estimation of the proposed scheme, where the truncation technique will be used to deal with the non-global Lipschitz nonlinearity.

\section{Convergence order in probability}
In this section, we focus on the approximate error for the proposed scheme in temporal direction. As the nonlinear term is not global Lipschitz, we consider the following truncated function first
\begin{align}\label{truncate}
d\Psi_R-\bi\left(\frac1{h^2}A\Psi_R+\bi\alpha\Psi_R+\lambda F_R(\Psi_R)\Psi_R\right)dt=\epsilon\sigma\Lambda d\beta,
\end{align}
with $\Psi_R:=\Psi_R(t)=(\psi_{R,1}^{}(t),\cdots,\psi_{R,J}^{}(t))^T$ and initial value $\Psi_R(0)=\Psi(0)$. Here $F_R(v)=\theta\left(\frac{\|v\|}{R}\right)F(v)$ for any vector $v\in\C^J$ and a cut-off function $\theta\in C^{\infty}(\R)$ satisfying $\theta(x)=1$ for $x\in[0,1]$ and $\theta(x)=0$ for $x\ge2$ (see also \cite{BD06,L13}). In addition, we have
\begin{align*}
\|F_R(\Psi_R)\|=\theta\left(\frac{\|\Psi_R\|}{R}\right)\max_{1\le j\le J}|\psi_{R,j}^{}|^2
\le\theta\left(\frac{\|\Psi_R\|}{R}\right)\|\Psi_R\|^2
\le4R^2.
\end{align*}
As a result, the nonlinear term $F_R(\Psi_R)\Psi_R$ is global Lipschitz with respect to the norm $\|\cdot\|$.
The proposed scheme \eqref{explicitscheme} applied to the truncated equation \eqref{truncate} yields the following scheme
\begin{equation}\label{truncatescheme}
\begin{split}
\Psi_R^{n+1}=&\left(I-\frac{\bi\tau}{2h^2}A+\frac14\alpha\tau I\right)^{-1}\left(I+\frac{\bi\tau}{2h^2}A-\frac14\alpha\tau I\right)e^{f_R(\Psi_R^n)}\Psi_R^n\\
&+\left(I-\frac{\bi\tau}{2h^2}A+\frac14\alpha\tau I\right)^{-1}\epsilon\sigma\Lambda\Delta_{n+1}\beta,
\end{split}
\end{equation}
where $f_R(\Psi_R^n)=\left(-\frac12\alpha I+\bi\lambda F_R(\Psi_R^n)\right)\tau$ and $\Psi_R^n=(\psi_{R,1}^n,\cdots,\psi_{R,J}^n)^T$.

\begin{tm}\label{trunerror}
Consider Eq. \eqref{truncate} and the scheme \eqref{truncatescheme}. Assume that $\E\|\psi_0\|_{L^2}^2<\infty$, $Q\in\mathcal{HS}(L^2,\dot{H}^2)$, $\alpha\ge\frac12$ and $\tau=O(h^4)$. For $T=N\tau$, there exists a constant $C_R$ which depends on $\alpha,\epsilon,R,Q,\psi_0$ and is independent of $T$ and $N$ such that
\begin{align*}
h\E\|\Psi_R(T)-\Psi_R^N\|^2\le C_R\tau^2.
\end{align*}
\end{tm}

\begin{proof}
Denote semigroup operator $S(t):=e^{Bt}$ which is generated by the linear operator $B:=\bi\frac1{h^2}A-\frac{\alpha}2I$, then the mild solution of \eqref{space} is
\begin{equation}\label{psit}
\begin{split}
\Psi_R(t_{n+1})=&S(\tau)\Psi(t_n)+\int_{t_n}^{t_{n+1}}S(t_{n+1}-s)\bi\lambda F_R(\Psi_R(s))\Psi_R(s)ds\\
&-\int_{t_n}^{t_{n+1}}S(t_{n+1}-s)\frac{\alpha}2\Psi_R(s)ds+\int_{t_n}^{t_{n+1}}S(t_{n+1}-s)\epsilon\sigma\Lambda d\beta(s).
\end{split}
\end{equation}
Subtracting \eqref{truncatescheme} from \eqref{psit}, we obtain
\begin{equation*}\label{strongerror}
\begin{split}
&\Psi_R(t_{n+1})-\Psi_R^{n+1}\\
=&S(\tau)\Psi_R(t_n)-\left(I-\frac12B\tau\right)^{-1}\left(I+\frac12B\tau\right)e^{f_R(\Psi_R^n)}\Psi_R^n\\
&+\int_{t_n}^{t_{n+1}}S(t_{n+1}-s)\bi\lambda F_R(\Psi_R(s))\Psi_R(s)ds-\int_{t_n}^{t_{n+1}}S(t_{n+1}-s)\frac{\alpha}2\Psi_R(s)ds\\
&+\int_{t_n}^{t_{n+1}}\left(S(t_{n+1}-s)-\left(I-\frac12B\tau\right)^{-1}\right)\epsilon\sigma\Lambda d\beta(s)\\
=&S(\tau)\left(\Psi_R(t_n)-\Psi_R^n\right)+\left[S(\tau)-\left(I-\frac12B\tau\right)^{-1}\left(I+\frac12B\tau\right)\right]\Psi_R^n\\
&+\left(I-\frac12B\tau\right)^{-1}\left(I+\frac12B\tau\right)\left(\Psi_R^n-e^{f_R(\Psi_R^n)}\Psi_R^n\right)\\
&+\int_{t_n}^{t_{n+1}}S(t_{n+1}-s)\bi\lambda F_R(\Psi_R(s))\Psi_R(s)ds-\int_{t_n}^{t_{n+1}}S(t_{n+1}-s)\frac12\alpha\Psi_R(s)ds\\
&+\int_{t_n}^{t_{n+1}}\left(S(t_{n+1}-s)-\left(I-\frac12B\tau\right)^{-1}\right)\epsilon\sigma\Lambda d\beta(s)\\
=&:\uppercase\expandafter{\romannumeral1}+\uppercase\expandafter{\romannumeral2}+\uppercase\expandafter{\romannumeral3}+\uppercase\expandafter{\romannumeral4}+\uppercase\expandafter{\romannumeral5}+\uppercase\expandafter{\romannumeral6}.
\end{split}
\end{equation*}
To show the strong convergence order of \eqref{truncatescheme}, we give the estimates of above terms, respectively.
For terms $\uppercase\expandafter{\romannumeral1}$ and $\uppercase\expandafter{\romannumeral2}$, we have
\begin{align}\label{term1}
\E\|\uppercase\expandafter{\romannumeral1}\|^2
=\E\left\|e^{(\bi\frac1{h^2}A-\frac{\alpha}2I)\tau}(\Psi_R(t_n)-\Psi_R^n)\right\|^2
=e^{-\alpha\tau}\E\|\Psi_R(t_n)-\Psi_R^n\|^2
\end{align}
and
\begin{equation}\label{term2}
\begin{split}
\E\|\uppercase\expandafter{\romannumeral2}\|^2
\le C\E\|(B\tau)^3\Psi_R^n\|^2
&\le C\tau^6\|B^3\|^2\E\|\Psi_R^n\|^2\\[2mm]
&\le Ch^{-13}\tau^6\|A\|^6
\le Ch^{-13}\tau^6
\end{split}
\end{equation}
 based on $\left|e^x-(1-\frac{x}2)^{-1}(1+\frac{x}2)\right|=O(x^3)$ as $x\to0$ and Proposition \ref{fullmoment}. In the last step of \eqref{term2}, we also used the fact that $\|A\|$ is uniformly bounded for any dimension $J$, whose proof is not difficult and is given in the Appendix for readers' convenience.
For term $\uppercase\expandafter{\romannumeral6}$,
Taylor expansion yields that
\begin{equation}
\begin{split}
\E\|\uppercase\expandafter{\romannumeral6}\|^2
&\le2\E\left\|\int_{t_n}^{t_{n+1}}\left(S(t_{n+1}-s)-S(\tau)\right)\epsilon\sigma\Lambda d\beta(s)\right\|^2\\
&+2\E\left\|\left(S(\tau)-\left(I-\frac12B\tau\right)^{-1}\right)\epsilon\sigma\Lambda\Delta_{n+1}\beta\right\|^2\\[3mm]
&\le C\tau^2\E\|B\epsilon\sigma\Lambda\Delta_{n+1}\beta\|^2
\le Ch^{-1}\tau^3.
\end{split}
\end{equation}
It then remains to estimate terms $\uppercase\expandafter{\romannumeral3}$, $\uppercase\expandafter{\romannumeral4}$ and $\uppercase\expandafter{\romannumeral5}$.
We can obtain the following equation in the same way as that of \eqref{moment}
\begin{align*}
\|\Psi_R^{n+1}\|^2-e^{-\alpha\tau}\|\Psi_R^n\|^2\le C\tau\|h^{-2}A\epsilon\sigma\Lambda\Delta_{n+1}\beta\|^2+C\|\epsilon\sigma\Lambda\Delta_{n+1}\beta\|^2.
\end{align*}
Multiplying above equation by $\|\Psi_R^{n+1}\|^2$, we get
\begin{equation*}
\begin{split}
&\|\Psi_R^{n+1}\|^4+\left(\|\Psi_R^{n+1}\|^2-e^{-\alpha\tau}\|\Psi_R^n\|^2\right)^2-e^{-2\alpha\tau}\|\Psi_R^n\|^4\\
\le& C\tau\left(\|\Psi_R^{n+1}\|^2-e^{-\alpha\tau}\|\Psi_R^n\|^2\right)\|h^{-2}A\epsilon\sigma\Lambda\Delta_{n+1}\beta\|^2+C\tau e^{-\alpha\tau}\|\Psi_R^n\|^2\|h^{-2}A\epsilon\sigma\Lambda\Delta_{n+1}\beta\|^2\\
&+C\left(\|\Psi_R^{n+1}\|^2-e^{-\alpha\tau}\|\Psi_R^n\|^2\right)\|\epsilon\sigma\Lambda\Delta_{n+1}\beta\|^2+C e^{-\alpha\tau}\|\Psi_R^n\|^2\|\epsilon\sigma\Lambda\Delta_{n+1}\beta\|^2\\
\le& \left(\|\Psi_R^{n+1}\|^2-e^{-\alpha\tau}\|\Psi_R^n\|^2\right)^2
+\tau e^{-2\alpha\tau}\|\Psi_R^n\|^4+C\tau\|h^{-2}A\epsilon\sigma\Lambda\Delta_{n+1}\beta\|^4+\frac{C}{\tau}\|\epsilon\sigma\Lambda\Delta_{n+1}\beta\|^4.
\end{split}
\end{equation*}
Based on \eqref{sto1} and \eqref{sto2}, we take expectation of above equation and derive
\begin{align*}
\E\|\Psi_R^{n+1}\|^4
&\le(1+\tau)e^{-2\alpha\tau}\E\|\Psi_R^n\|^4+Ch^{-2}\tau\\[2mm]
&\le(1+\tau)^{n+1}e^{-2\alpha\tau (n+1)}\E\|\Psi_R^0\|^4+Ch^{-2}
\le Ch^{-2}
\end{align*}
for $\alpha\ge\frac12$, where in the last step we have used the fact that $\E\|\Psi_R^0\|^4\le(\E\|\Psi_R^0\|^2)^2\le Ch^{-2}$ and $(1+\tau)e^{-2\alpha\tau}<1$ for $\alpha\ge\frac12$.
Similarly, we derive
$\E\|\Psi_R^n\|^8\le Ch^{-4},$
which implies that
\begin{align*}
\E\|F_R(\Psi_R^n)\|^4=\E\left(\sum_{j=1}^J\left|\psi^n_{R,j}\right|^4\right)^2\le\E\|\Psi_R^n\|^8\le Ch^{-4},\quad\forall~n\in\N.
\end{align*}
Thus, by Taylor expansion, we have
\begin{equation}\label{345}
\begin{split}
&\uppercase\expandafter{\romannumeral3}+\uppercase\expandafter{\romannumeral4}+\uppercase\expandafter{\romannumeral5}=\left(I-\frac12B\tau\right)^{-1}\left(I+\frac12B\tau\right)\Big{(}-f_R(\Psi_R^n)+O(f(\Psi_R^n)^2)\Big{)}\Psi_R^n\\
&+\int_{t_n}^{t_{n+1}}S(t_{n+1}-s)\bi\lambda F_R(\Psi_R(s))\Psi_R(s)ds
-\int_{t_n}^{t_{n+1}}S(t_{n+1}-s)\frac12\alpha\Psi_R(s)ds\\
=&\bi\lambda \int_{t_n}^{t_{n+1}}\left[S(t_{n+1}-s)F_R(\Psi_R(s))\Psi_R(s)-\left(I-\frac12B\tau\right)^{-1}\left(I+\frac12B\tau\right)F_R(\Psi_R^n)\Psi_R^n\right]ds\\
&-\frac12\alpha\int_{t_n}^{t_{n+1}}\left[S(t_{n+1}-s)\Psi_R(s)-\left(I-\frac12B\tau\right)^{-1}\left(I+\frac12B\tau\right)\Psi_R^n\right]ds\\
&+\left(I-\frac12B\tau\right)^{-1}\left(I+\frac12B\tau\right)O(f_R(\Psi_R^n)^2)\Psi_R^n\\
:=&\tilde{\uppercase\expandafter{\romannumeral3}}+\tilde{\uppercase\expandafter{\romannumeral4}}+\tilde{\uppercase\expandafter{\romannumeral5}}.
\end{split}
\end{equation}
Now we estimate above terms respectively. For $\tilde{\uppercase\expandafter{\romannumeral3}}$, we have
\begin{align*}
\tilde{\uppercase\expandafter{\romannumeral3}}
=& \bi\lambda\int_{t_n}^{t_{n+1}}\left[S(t_{n+1}-s)-\left(I-\frac12B\tau\right)^{-1}\left(I+\frac12B\tau\right)\right]F_R(\Psi_R^n)\Psi_R^nds\nonumber\\
&+ \bi\lambda\int_{t_n}^{t_{n+1}}S(t_{n+1}-s)\Big{[}F_R(\Psi_R(s))F_R(\Psi_R(s))-F_R(\Psi_R^n)\Psi_R^n\Big{]}ds\nonumber\\
=:&\tilde{\uppercase\expandafter{\romannumeral3}}_1+\tilde{\uppercase\expandafter{\romannumeral3}}_2,
\end{align*}
which satisfies
\begin{align*}
\E\|\tilde{\uppercase\expandafter{\romannumeral3}}_1\|^2
\le& \tau\int_{t_n}^{t_{n+1}}\E\left\|\left[S(t_{n+1}-s)-\left(I-\frac12B\tau\right)^{-1}\left(I+\frac12B\tau\right)\right]F_R(\Psi_R^n)\Psi_R^n\right\|^2ds\\
\le& Ch^{-12}\tau^8\E\|F_R(\Psi_R^n)\Psi_R^n\|^2
\le Ch^{-12}\tau^8\E\left[\sum_{j=1}^J\left|\psi_{R,j}^n\right|^6\right]
\le Ch^{-15}\tau^8
\end{align*}
and
\begin{align*}
\E\|\tilde{\uppercase\expandafter{\romannumeral3}}_2\|^2
\le& \tau\int_{t_n}^{t_{n+1}}\E\Big{\|}S(t_{n+1}-s)\Big{[}F_R(\Psi_R(s))F_R(\Psi_R(s))-F_R(\Psi_R^n)\Psi_R^n\Big{]}\Big{\|}^2ds\\
\le&C\tau\int_{t_n}^{t_{n+1}}\E\|\Psi_R(s)-\Psi_R(t_n)\|^2ds+C\tau^2e^{-\alpha\tau}\E\|\Psi_R(t_n)-\Psi_R^n\|^2.
\end{align*}
Noticing that
\begin{align*}
\E\|\Psi_R(s)-\Psi_R(t_n)\|^2
=&\E\bigg{\|}\int_{t_n}^sS(s-r)\bi\lambda F_R(\Psi_R(r))\Psi_R(r)dr-\int_{t_n}^sS(s-r)\frac{\alpha}2\Psi_R(r)dr\\
&+\int_{t_n}^sS(s-r)\epsilon\sigma\Lambda d\beta(r)\bigg{\|}^2
\le h^{-3}\tau^2,
\end{align*}
thus
\begin{align}
\E\|\tilde{\uppercase\expandafter{\romannumeral3}}\|^2
\le Ch^{-15}\tau^8+Ch^{-3}\tau^4+C\tau^2e^{-\alpha\tau}\E\|\Psi_R(t_n)-\Psi_R^n\|^2.
\end{align}
Term $\tilde{\uppercase\expandafter{\romannumeral4}}$ can be estimated in the same way as the estimation of $\tilde{\uppercase\expandafter{\romannumeral3}}$. Term $\tilde{\uppercase\expandafter{\romannumeral5}}$ turns to be
\begin{equation}\label{term5}
\begin{split}
\E\|\tilde{\uppercase\expandafter{\romannumeral5}}\|^2
\le& C\E\left\|f_R(\Psi_R^n)^2\Psi_R^n\right\|^2
\le C\tau^4\E\left[\sup_{1\le j\le J}\left|-\frac12\alpha+\bi\lambda|\psi_{R,j}^n|^2\right|^4\|\Psi_R^n\|^2\right]\\
\le& C\tau^4\left(\E\left(\sum_{j=1}^J\left|\psi^n_{R,j}\right|^2\right)^8\right)^{\frac12}\left(\E\|\Psi_R^n\|^4\right)^{\frac12}
\le Ch^{-5}\tau^4.
\end{split}
\end{equation}
From \eqref{term1}--\eqref{term5}, we conclude
\begin{align*}
&h\E\|\Psi_R(t_{n+1})-\Psi_R^{n+1}\|^2\\[2mm]
\le& h(1+C\tau^2)e^{-\alpha\tau}\E\|\Psi_R(t_{n})-\Psi_R^{n}\|^2+C\tau^3+Ch^{-4}\tau^4+Ch^{-12}\tau^6+Ch^{-14}\tau^8\\[2mm]
\le& C\tau^{2}+Ch^{-4}\tau^{3}+Ch^{-12}\tau^5+Ch^{-14}\tau^7
\le C\tau^2,
\end{align*}
where in the last two steps we have used the fact that $(1+C\tau^2)e^{-\alpha\tau}<1$ for $\tau$ is sufficiently small and $\tau=O(h^4)$.
\end{proof}
Based on the estimates on truncated equation and its numerical scheme, we are now in the position to give the approximate error between $\Psi(t)$ and $\Psi^n$. The proof of following theorem is motivated by \cite{BD06,L13} and holds for any fixed $T>0$ without other restrictions.
\begin{tm}\label{probability}
Consider Eq. \eqref{space} and scheme \eqref{explicitscheme}. Assume that  $\E\|\psi_0\|_{L^2}^2<\infty$, $Q\in\mathcal{HS}(L^2,\dot{H}^2)$, $\alpha\ge\frac12$ and $\tau=O(h^4)$. For any $T>0$, we derive convergence order one in probability, i.e.,
\begin{align}
\lim_{K\to\infty}\P\left(\sup_{1\le n\le[T/\tau]}\sqrt{h}\|\Psi(t_{n})-\Psi^{n}\|\ge K\tau\right)=0.
\end{align}
\end{tm}

\begin{proof}
For any $\gamma\in(0,1)$,
we define
$n_{\gamma}:=\inf\{1\le n\le[T/\tau]:\|\Psi(t_n)-\Psi^n\|\ge\gamma\}$ and then deduce that
\begin{align*}
&\left\{\sup_{1\le n\le[T/\tau]}\|\Psi(t_{n})-\Psi^{n}\|\ge\gamma\right\}\\
\subset&\Bigg{[}\left(\left\{\sup_{0\le n\le n_\gamma}\|\Psi(t_n)\|\ge R-1\right\}\cap\left\{\sup_{1\le n\le[T/\tau]}\|\Psi(t_{n})-\Psi^{n}\|\ge\gamma\right\}\right)\\
&\cup\left(\left\{\sup_{0\le n\le n_\gamma}\|\Psi(t_n)\|<R-1\right\}\cap\left\{\sup_{1\le n\le[T/\tau]}\|\Psi(t_{n})-\Psi^{n}\|\ge\gamma\right\}\right)\Bigg{]}\\
\subset&\Bigg{[}\left\{\sup_{0\le n\le n_\gamma}\|\Psi(t_n)\|\ge R-1\right\}\\
&\cup
\left(\left\{\sup_{0\le n\le n_\gamma}\|\Psi(t_n)\|<R-1\right\}\cap\left\{\sup_{1\le n\le[T/\tau]}\|\Psi(t_{n})-\Psi^{n}\|\ge\gamma\right\}\right)\Bigg{]}.
\end{align*}
If $\left\{\sup_{0\le n\le n_\gamma}\|\Psi(t_n)\|<R-1\right\}$ happens, it is easy to show that
$\|\Psi^k\|\le\|\Psi(t_k)-\Psi^k\|+\|\Psi(t_k)\|<R-1+\gamma<R$, $F_R(\Psi_R^k)=F(\Psi_R^k)$, $\Psi_R^k=\Psi^k$ for $k=0,1,\cdots,n_\gamma-1$ and $\Psi_R(t_n)=\Psi(t_n)$ for $0\le n\le n_\gamma$.
Furthermore, comparing scheme \eqref{truncatescheme} with \eqref{explicitscheme} and noticing that
\begin{equation}
\begin{split}
f_R(\Psi_R^{n_\gamma-1})&=\left(-\frac12\alpha I+\bi\lambda F_R(\Psi_R^{n_\gamma-1})\right)\tau\\[2mm]
&=\left(-\frac12\alpha I+\bi\lambda F(\Psi^{n_\gamma-1})\right)\tau=f(\Psi^{n_\gamma-1}),
\end{split}
\end{equation}
we have $\Psi_R^{n_\gamma}=\Psi^{n_\gamma}$,
which implies
$$\|\Psi_R(t_{n_\gamma})-\Psi_R^{n_\gamma}\|=\|\Psi(t_{n_\gamma})-\Psi^{n_\gamma}\|\ge\gamma.$$
We conclude that for any $\gamma\in(0,1)$, there exists $n_\gamma\in\N$ such that
\begin{align*}
&\left\{\sup_{0\le n\le n_\gamma}\|\Psi(t_n)\|<R-1\right\}\cap\left\{\sup_{1\le n\le[T/\tau]}\|\Psi(t_{n})-\Psi^{n}\|\ge\gamma\right\}\\
\subset&\{\|\Psi_R(t_{n_\gamma})-\Psi_R^{n_\gamma}\|\ge\gamma\}.
\end{align*}
Thus, for some constants $K,K_1>0$, choosing $\gamma=\sqrt{h^{-1}}K\tau$ and $R=\sqrt{h^{-1}}K_1$,  we deduce
\begin{equation}\label{error}
\begin{split}
&\P\left(\sup_{1\le n\le[T/\tau]}\sqrt{h}\|\Psi(t_{n})-\Psi^{n}\|\ge K\tau\right)\\
\le&\P\left(\sup_{0\le n\le n_\gamma}\sqrt{h}\|\Psi(t_n)\|\ge K_1\right)+\P\left(\sqrt{h}\|\Psi_R(t_{n_\gamma})-\Psi_R^{n_\gamma}\|\ge K\tau\right)\\
\le&\frac{h\E\Big[\sup\limits_{0\le n\le n_{\gamma}}\|\Psi(t_n)\|^2\Big]}{K_1^2}+\frac{h\E\|\Psi_R(t_{n_\gamma})-\Psi_R^{n_\gamma}\|^2}{K^2\tau^2}.
\end{split}
\end{equation}
We claim that $e^{2\alpha t}\|\Psi(t)\|^2$ is a submartingale, which ensures that
\begin{equation*}
\begin{split}
h\E\left[\sup_{0\le n\le n_{\gamma}}\|\Psi(t_n)\|^2\right]
&\le h\E\left[\sup_{0\le n\le n_{\gamma}}e^{2\alpha t_n}\|\Psi(t_n)\|^2\right]\\[2mm]
&\le e^{2\alpha T}h\E\left[\|\Psi(t_{n_\gamma})\|^2\right]
\le Ce^{2\alpha T}
\end{split}
\end{equation*}
based on a martingale inequality and Proposition \ref{semimoment}.
In fact, denoting $C_{J,P}:=\sum_{j=1}^J\sum_{k=1}^P\eta_ke_k^2(x_j)$ and
applying It\^o's formula to $e^{2\alpha t}\|\Psi(t)\|^2$ similar to \eqref{semiito},
we derive
\begin{align*}
e^{2\alpha t}\|\Psi(t)\|^2=\|\Psi(0)\|^2+2\int_0^te^{2\alpha s}\Re\left[\overline{\Psi}(s)\epsilon\sigma\Lambda d\beta(s)\right]+\frac{C_{J,P}}{\alpha}\left(e^{2\alpha t}-1\right)
\end{align*}
with $2\int_0^te^{2\alpha s}\Re\left[\overline{\Psi}(s)\epsilon\sigma\Lambda d\beta(s)\right]$ a martingale.
Apparently, we have
\begin{align*}
\E\left[e^{2\alpha t}\|\Psi(t)\|^2|\mathcal{F}_r\right]
&=\|\Psi(0)\|^2+2\int_0^re^{2\alpha s}\Re\left[\overline{\Psi}(s)\epsilon\sigma\Lambda d\beta(s)\right]+\frac{C_{J,P}}{\alpha}\left(e^{2\alpha t}-1\right)\\[2mm]
&\ge\|\Psi(0)\|^2+2\int_0^re^{2\alpha s}\Re\left[\overline{\Psi}(s)\epsilon\sigma\Lambda d\beta(s)\right]+\frac{C_{J,P}}{\alpha}\left(e^{2\alpha r}-1\right)\\[2mm]
&=e^{2\alpha r}\|\Psi(r)\|^2
\end{align*}
for $r\le t$, which completes the claim.
Hence, based on above claim and Theorem \ref{trunerror},  inequality \eqref{error} turns to be
\begin{align*}
\P\left(\sup_{1\le n\le[T/\tau]}\sqrt{h}\|\Psi(t_{n})-\Psi^{n}\|\ge K\tau\right)
\le\frac{Ce^{2\alpha T}}{K_1^2}+\frac{C_R}{K^2},
\end{align*}
which approaches to $0$ as $K_1,K\to+\infty$ properly for any $T>0$.
\end{proof}

\section{Numerical experiments}
In this section, we provide several numerical experiments to illustrate the accuracy and capability of the fully discrete scheme \eqref{full}, which can be calculated explicitly. We investigate the good performance in longtime simulation of the proposed scheme and check the temporal accuracy by
fixing the space step.
In the sequel, we take $\lambda=1,\,\alpha=0.5$, truncate the infinite series of Wiener process till $P=100$ and choose 500 realizations to approximate the expectation.

\begin{figure}[H]
\centering
\subfigure[$\epsilon=0$]{
\begin{minipage}[t]{0.48\linewidth}
  \includegraphics[height=6.0cm,width=6.0cm]{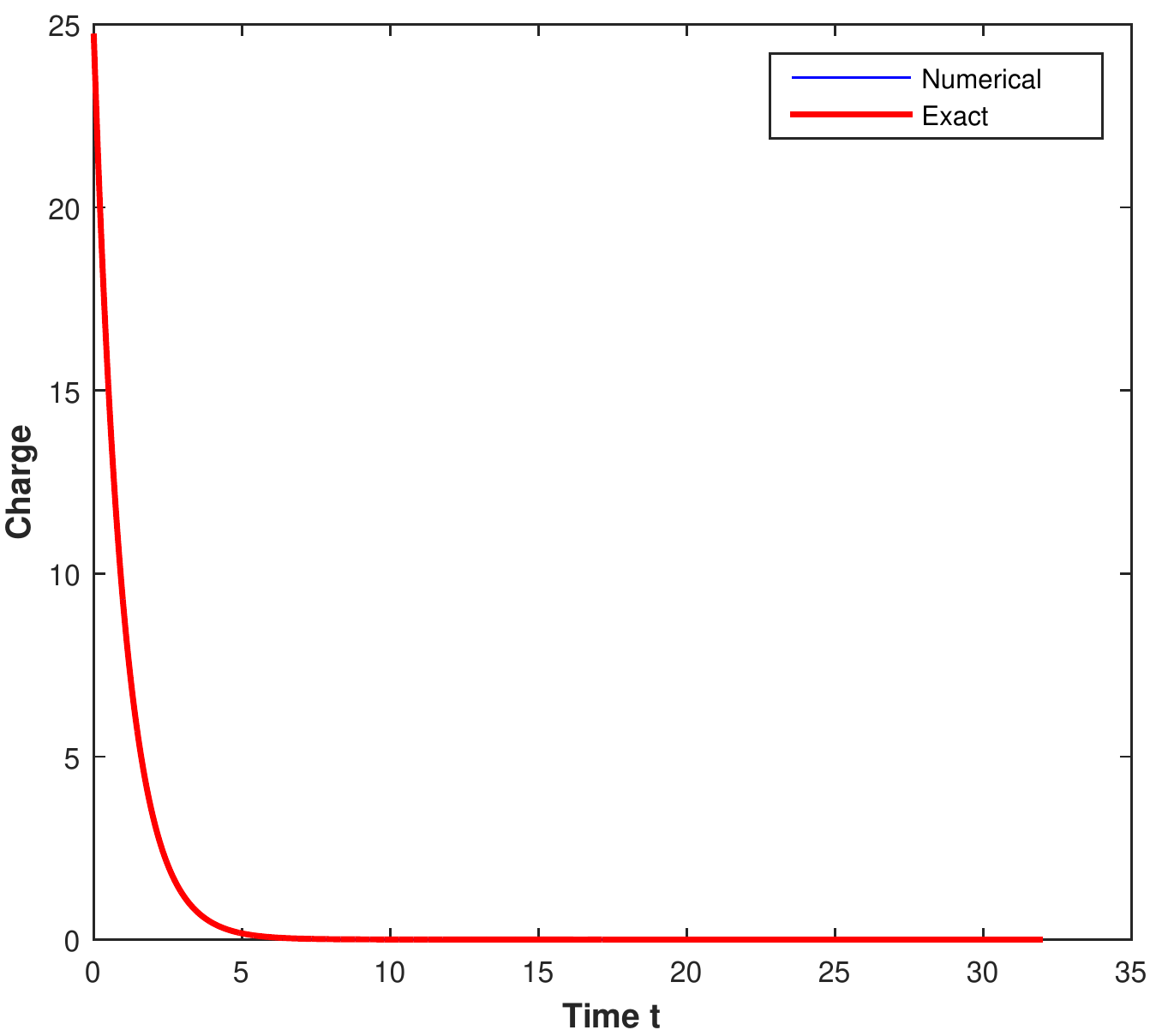}
  \end{minipage}
  }
  \subfigure[$\epsilon=1$]{
  \begin{minipage}[t]{0.48\linewidth}
  \includegraphics[height=6.0cm,width=6.0cm]{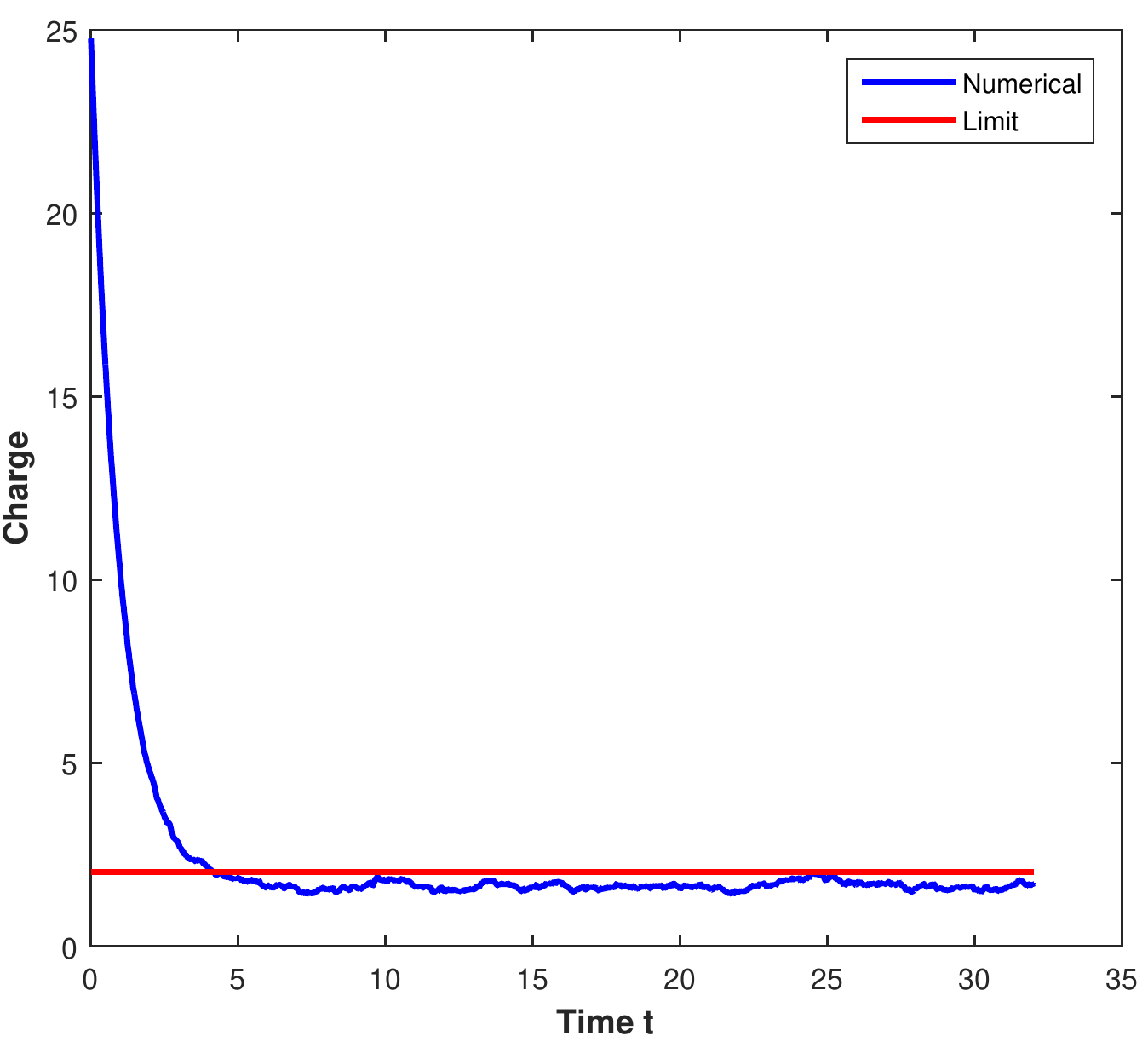}
  \end{minipage}
  }
  \caption{Evolution of the discrete charge $h\E\|\Psi^n\|^2$ with $t=n\tau$ for (a) $\epsilon=0$ and (b) $\epsilon=1$ ($h=0.1,\, \tau=2^{-6},\, T=35$).}\label{charge00}
\end{figure}

{\it Charge evolution.} For the semidiscretization, the charge of the solution satisfies the evolution formula \eqref{charge}. To investigate the recurrence relation for the discrete charge of the fully discrete scheme, Fig. \ref{charge00} plots the discrete charge for different values of $\epsilon$ with initial value $\psi_0(x)=\sin(\pi x)$, $\eta_k=k^{-6}$, $h=1/J=0.1$, $\tau=2^{-5}$ and $T=35$. We can observe that the discrete charge inherits the charge dissipation law without the noise term, i.e., $\epsilon=0$, and preserves the charge dissipation law approximately with a limit $\frac{\epsilon^2h}{\alpha}\sum_{j=1}^J\sum_{k=1}^P\eta_ke_k^2(x_j)$ calculated through \eqref{charge} for $\epsilon=1$.

\begin{figure}[H]
\centering
\subfigure[$f=\exp(-\|\Psi\|^2)$]{
\begin{minipage}[t]{0.48\linewidth}
  \includegraphics[height=6.0cm,width=6.0cm]{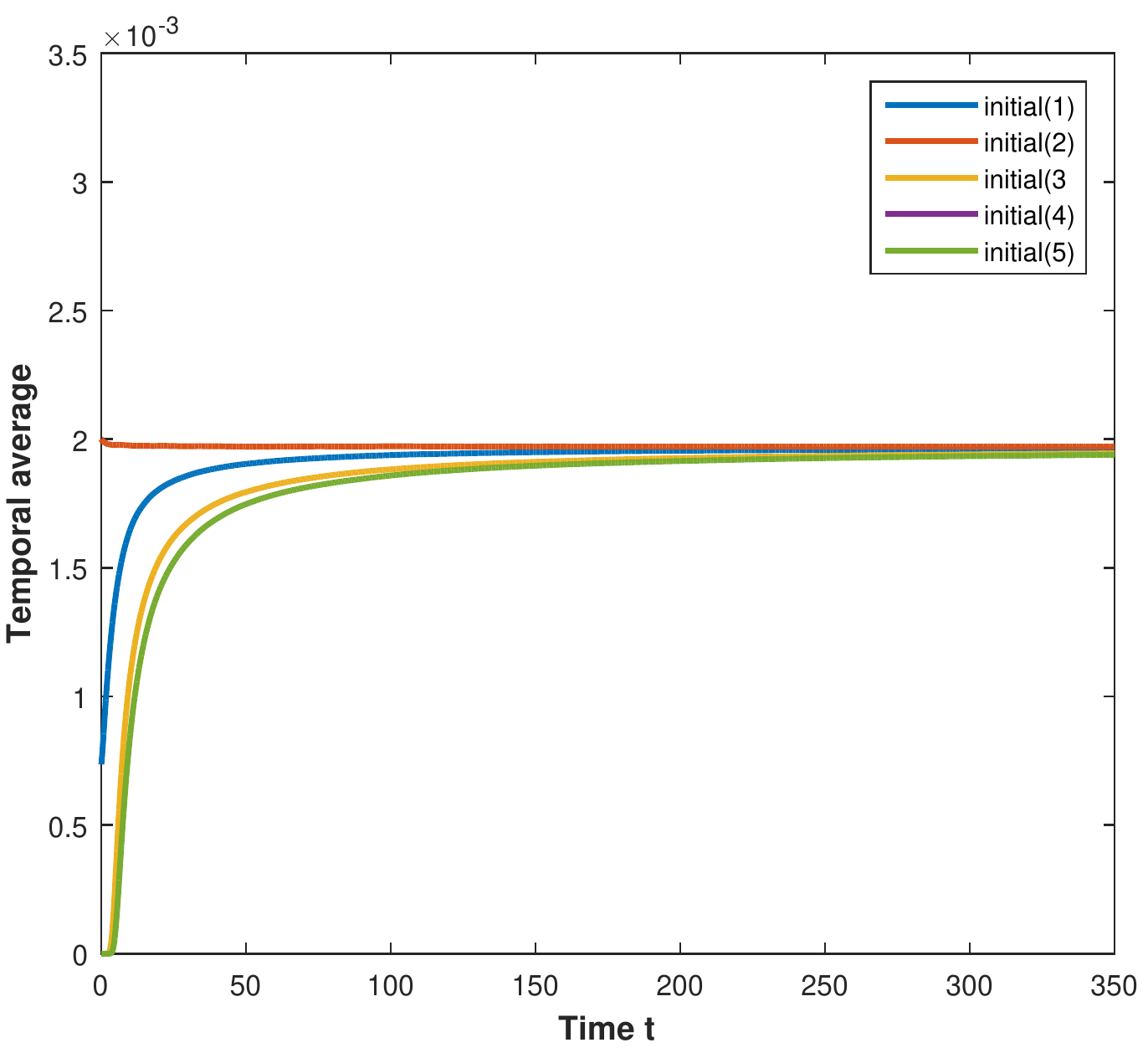}
\end{minipage}
}
\subfigure[$f=\sin(\|\Psi\|^2)$]{
\begin{minipage}[t]{0.48\linewidth}
  \includegraphics[height=6.0cm,width=6.0cm]{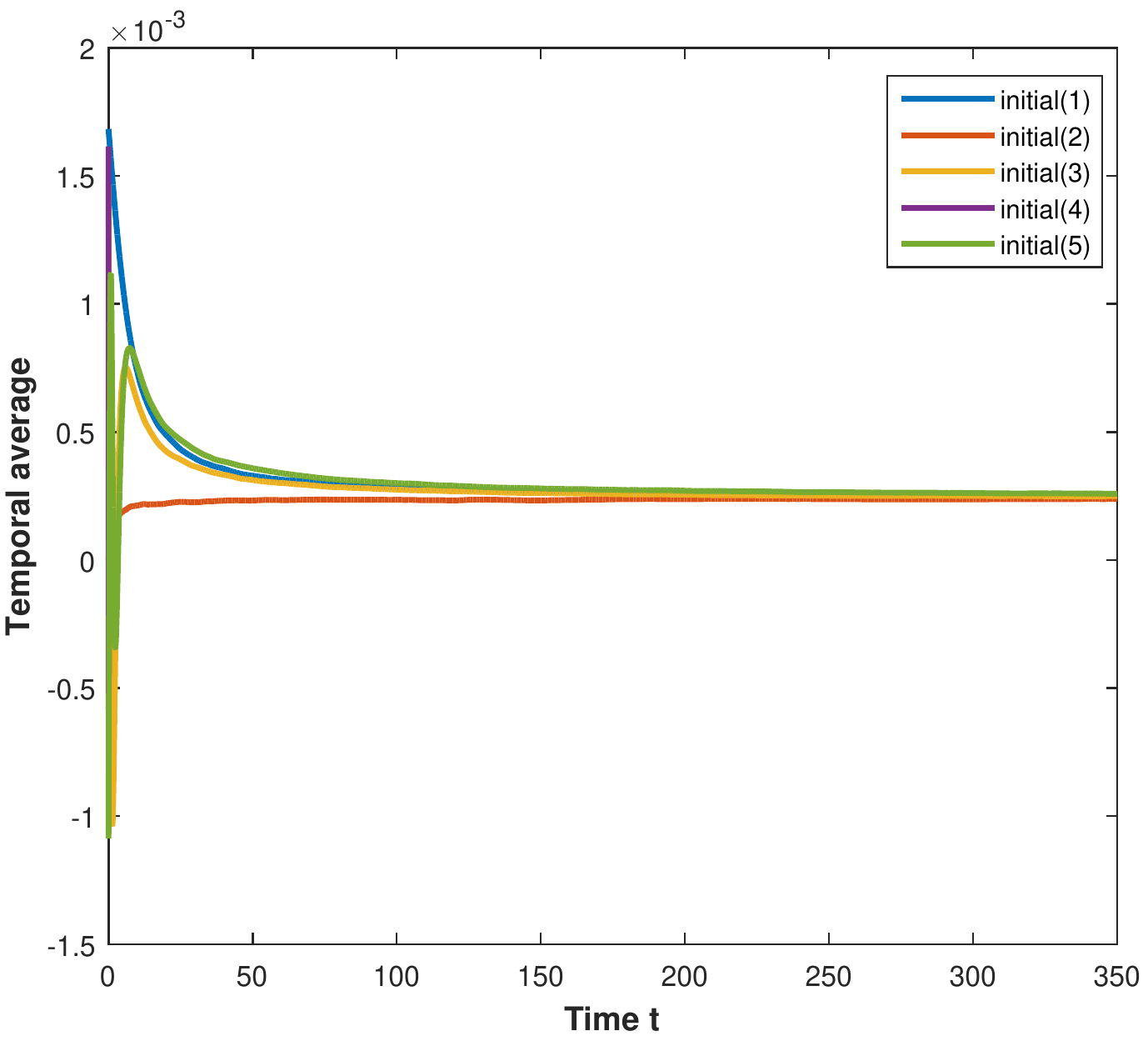}
\end{minipage}
}
  \caption{The temporal averages $\frac{1}{N}\sum_{n=1}^{N-1}\E[f(\Psi^n)]$ started from different initial values for bounded functions (a) $f=\exp(-\|\Psi\|^2)$ and (b) $f=\sin(\|\Psi\|^2)$ ($h=0.1,\,\epsilon=1,\,\tau=2^{-6},\,T=350$).}\label{temporal_average}
\end{figure}

{\it Ergodicity.}
Based on the definition of ergodicity, if numerical solution $\Psi^n$ is ergodic, its temporal averages $\frac{1}{N}\sum_{n=1}^{N-1}\E[f(\Psi^n)]$ started from different initial values will converge to the spatial average $\int_{\C^J}fd\mu_h^\tau$.
To verify this property, Fig. \ref{temporal_average} shows the
 temporal averages of the fully discrete scheme started from five different initial values $initial(1)=(1,~0,~\cdots,~0)^{T}$, $initial(2)=(0.0003 \bi,~0,~\cdots,~0)^{T}$, $initial(3)=(\sin\Big(\frac{1}{101}\pi\Big),~\sin\Big(\frac{2}{101}\pi\Big),~\cdots,~\sin\Big(\frac{100}{101}\pi\Big))^{T}$, $initial(4)=\frac{2+\bi}{20}(1,~2,~\cdots,~100)^{T}$ and $initial(5)=(\exp(-\frac{\bi}{50}),~\exp(-\frac{2\bi}{50}),~\cdots,~\exp(-\frac{100\bi}{50}))^{T}$. From Fig. \ref{temporal_average}, it can be seen that the time averages of \eqref{full} started from different initial values converge to the same value for two continuous and bounded functions $f$, when time $T=250$ is sufficiently large.

\begin{figure}[H]
\centering
\subfigure[$\tau=2^{-8}$]{
\begin{minipage}[t]{0.48\linewidth}
  \includegraphics[height=6.0cm,width=6.0cm]{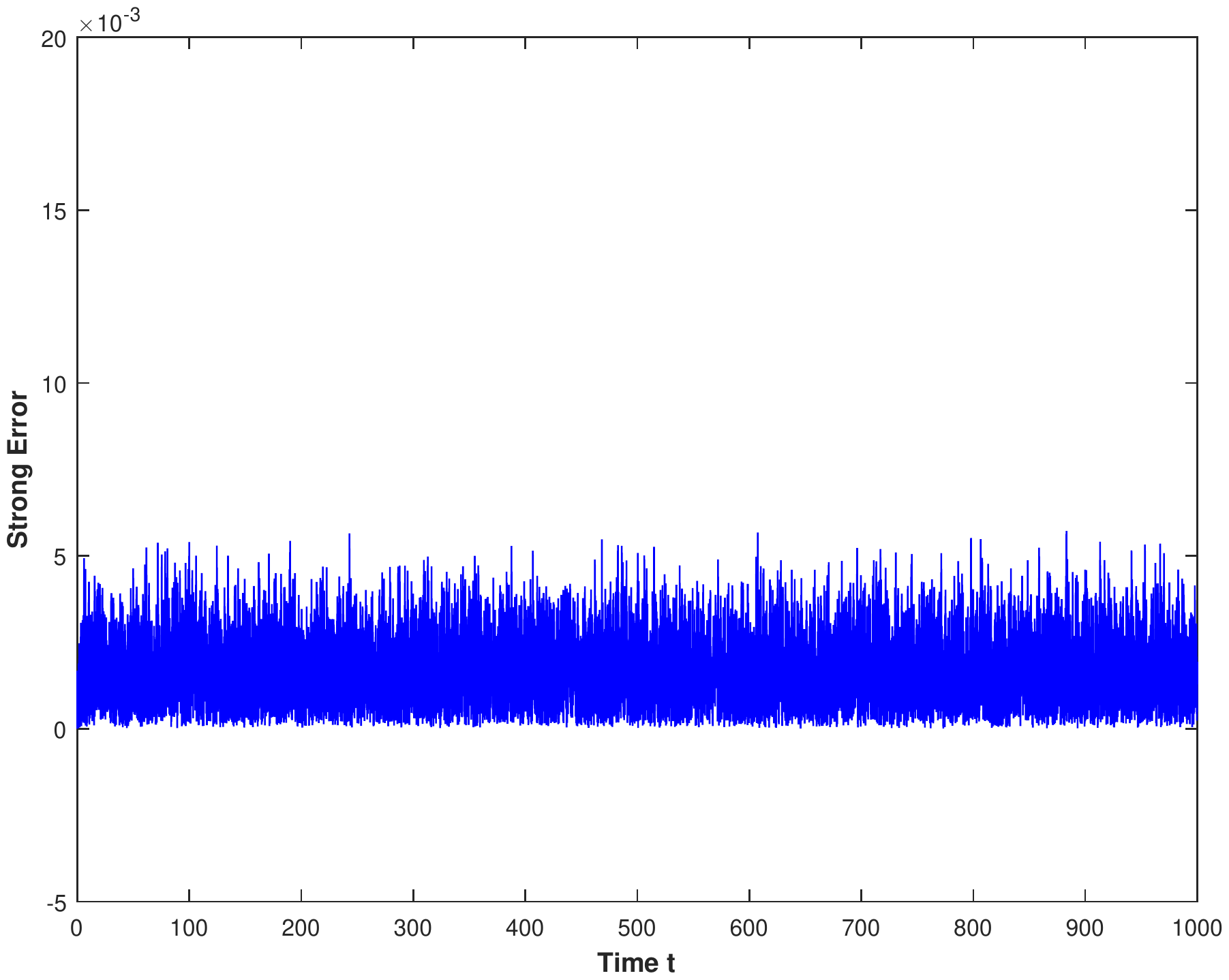}
 \end{minipage}
 }
 \subfigure[$\tau=2^{-10}$]{
\begin{minipage}[t]{0.48\linewidth}
  \includegraphics[height=6.0cm,width=6.0cm]{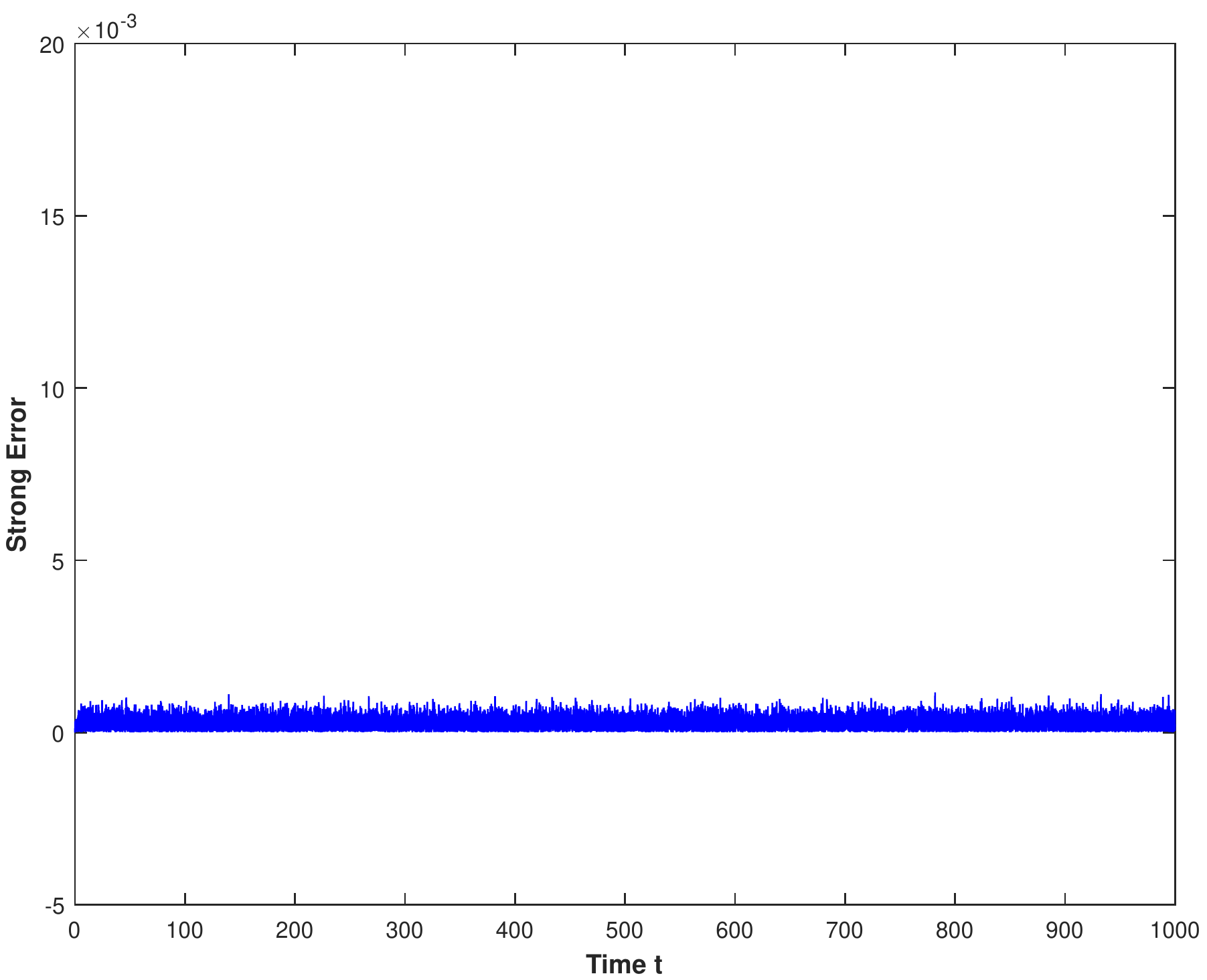}
  \end{minipage}
  }
     \caption{The mean-square convergence error $\Big(h\mathbb{E}\|\Psi(T)-\Psi^N\|^{2}\Big)^{\frac{1}{2}}$ for step sizes (a) $\tau=2^{-8}$ and (b) $\tau=2^{-10}$ ($h=0.25,\,\epsilon=1,T=10^3$). }\label{strong_error}
\end{figure}

{\it Time-independent error.}
As stated in Theorem \ref{trunerror} and \ref{probability}, the mean-square convergence error $\left(h\E\|\Psi_R(T)-\Psi_R^N\|^2\right)^{\frac12}$ with respect to the truncated equation \eqref{truncate} is independent of time $T$, and convergence in probability sense with respect to the original equation is also independent of time $T$. To clarify this property, by defining the mean-square convergence error as
\begin{equation*}
  \mathcal{E}_{h,\tau}:=\Big(h\mathbb{E}\|\Psi(T)-\Psi^N\|^{2}\Big)^{\frac{1}{2}},~T=N\tau,
\end{equation*}
Fig. \ref{strong_error} displays the error $\mathcal{E}_{h,\tau}$ over long time $T=10^3$ for different time step sizes: (a) $\tau=2^{-8}$ and (b) $\tau=2^{-10}$ with $h=0.25$, and shows that the mean-square convergence error is independent of time interval, which coincides with our theoretical results.

\begin{figure}[H]
\centering
 \subfigure[$\epsilon=0$]{
\begin{minipage}[t]{0.48\linewidth}
  \includegraphics[height=6.0cm,width=6.0cm]{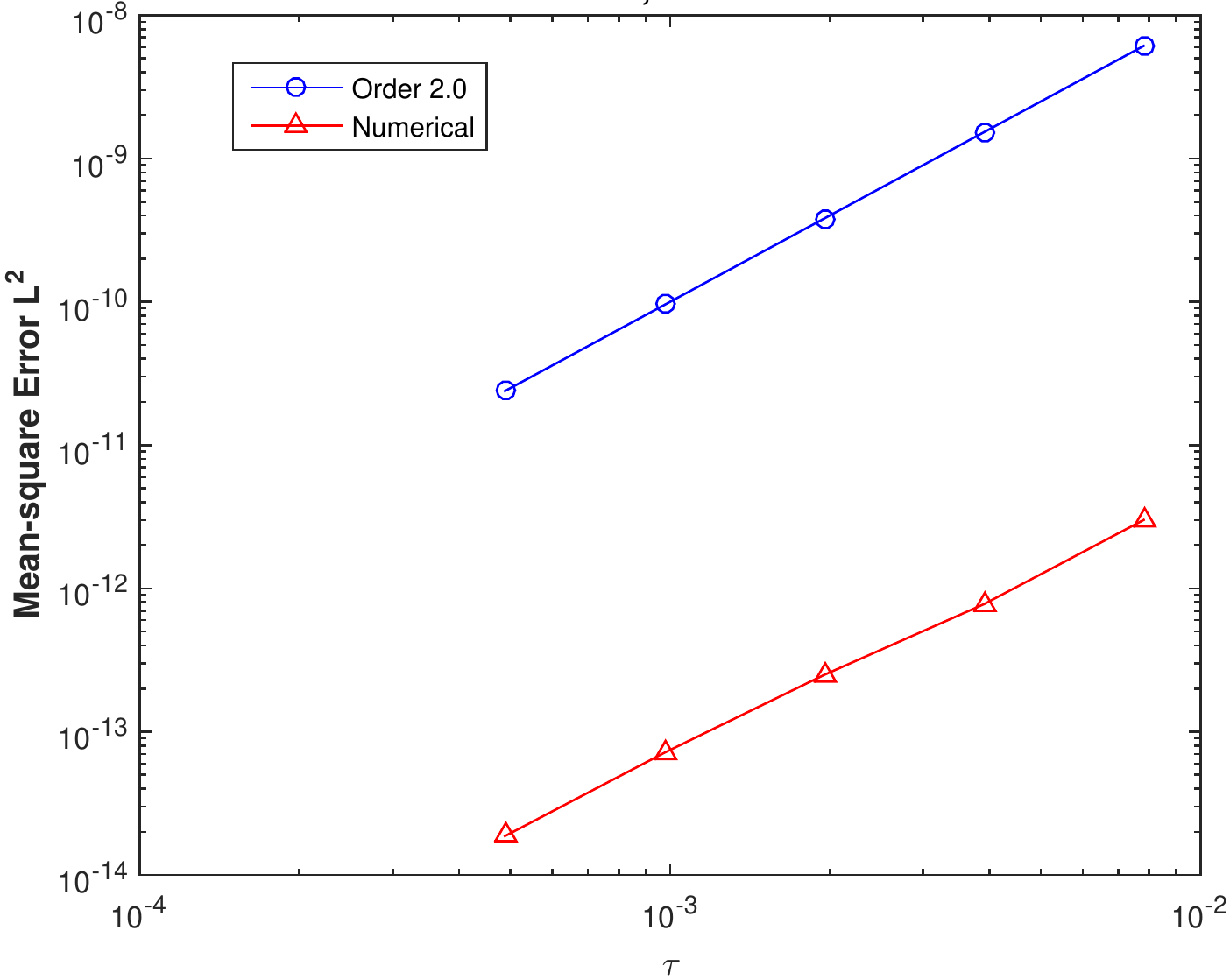}
 \end{minipage}
 }
  \subfigure[$\epsilon=1$]{
\begin{minipage}[t]{0.48\linewidth}
  \includegraphics[height=6.0cm,width=6.0cm]{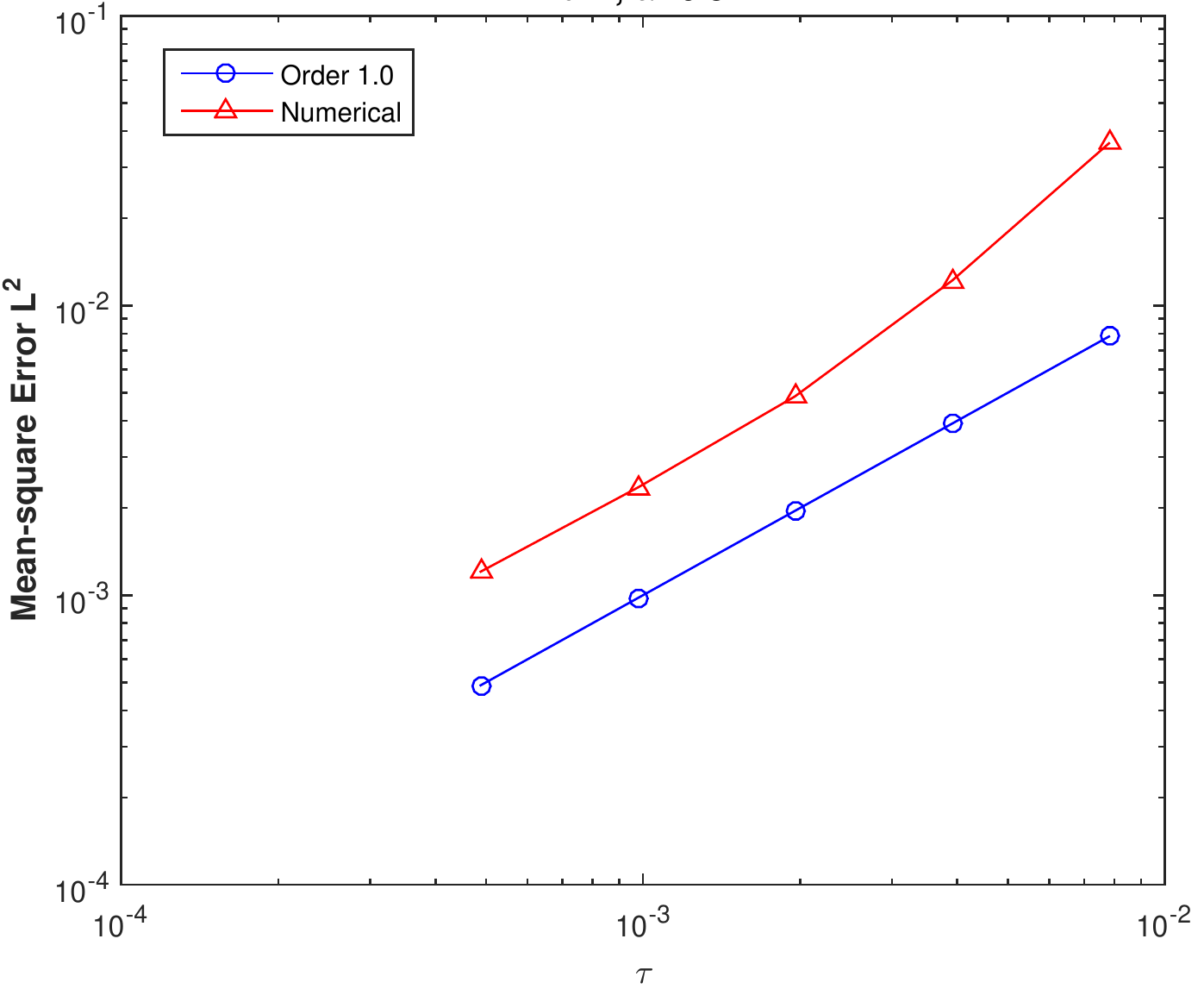}
  \end{minipage}
  }
     \caption{Rates of convergence of \eqref{full} for (a) $\epsilon=0$ and (b) $\epsilon=1$, respectivly ($h=0.1,\,T=1,\,\tau=2^{-l},\,11\le l\le14$). }\label{order}
\end{figure}

{\it Convergence order.}
We investigate the mean-square convergence order in temporal direction of the proposed method \eqref{full} in this experiment. Let $h=0.1$, $T=1$ and initial value $\psi_0(x)=\sin(\pi x)$. We plot
 $\mathcal{E}_{h,\tau}$ against $\tau$ on a log-log scale with various combinations of $(\alpha,\epsilon)$ and  take the method \eqref{full} with small time stepsize $\tau=2^{-16}$ as the reference solution. We then compare it to
the method \eqref{full} evaluated with time steps $(2^{2}\tau, 2^{3}\tau, 2^{4}\tau, 2^{5}\tau)$ in order to show the rate of convergence.
Fig. \ref{order} presents the mean-square convergence order for the error $\mathcal{E}_{h,\tau}$ with various sizes of $\epsilon$. Fig. \ref{order} shows that the proposed scheme \eqref{full} is of order 2 for the deterministic case, i.e., $\epsilon=0$, and of order 1 for the stochastic case with $\epsilon=1$, which coincides with the theoretical analysis.

\section{Appendix}
\begin{proof}[Proof of uniform boundedness of $\|A\|$]
Based on the definition of $\|A\|$, we only need to show that
the maximum eigenvalue $\lambda_*:=\max\{\lambda:\det(\lambda I-\hat{A})=0\}$ of positive definite matrix
$\hat{A}:=-A\in\R^{J\times J}$ is uniformly bounded with respect to dimension $J$. Let $x:=\lambda-2>-2$. Then
$\lambda I-\hat{A}$ turns to be
\begin{equation*}
\left(
\begin{array}{cccc}
x&1 & & \\
 1&x&1 & \\
&\ddots &\ddots &\ddots \\
 & &1 &x\\
\end{array}
\right)\\
\cong\left(
\begin{array}{cccc}
x&1 & & \\
 &x-\frac1x&1 & \\
& & x-\frac1{x-\frac1x}& \\
 & & &\ddots\\
\end{array}
\right)=:X.
\end{equation*}
We define
$a_1(x):=x$, $a_{n+1}(x):=x-\frac1{a_{n}(x)}$ and $X_n(x)=\prod_{i=1}^na_i(x)$ for $n\ge1$, and deduce that
\begin{align}\label{Xn}
X_{n+2}(x)=xX_{n+1}(x)-X_{n}(x).
\end{align}
Noticing that
$X_2(x)=x^2-1>0$ and $X_2(x)-X_1(x)=x^2-x-1>0$ for any $x\ge2$.
We assume that $X_{j+1}(x)>0$ and $X_{j+1}(x)-X_{j}(x)>0$
for any $x\ge2$ and $1\le j\le n$, which contributes to
\begin{align*}
X_{n+2}(x)-X_{n+1}(x)=(x-2)X_{n+1}(x)+(X_{n+1}(x)-X_{n}(x))>0
\end{align*}
and $X_{n+2}(x)>X_{n+1}(x)>0$
based on \eqref{Xn}. Then the induction yields that $X_n(x)>0$ for any $x\ge2$ and $n\in\N$, which implies that $\lambda_*=\max\{x:X_J(x)=0\}+2\le 4.$
\end{proof}

\nocite{*} 

\bibliography{wangxu}

\def\cprime{$'$}
\begin{thebibliography}{10}

\bibitem{abdulle}
A.~Abdulle, G.~Vilmart, and K.~C. Zygalakis.
\newblock {High order numerical approximation of the invariant measure of
  ergodic SDEs}.
\newblock {\em SIAM J. Numer. Anal.}, 52(4):1600--1622, 2014.

\bibitem{brehier}
C.-E. Br{\'e}hier.
\newblock Approximation of the invariant measure with an {E}uler scheme for
  stochastic {PDE}s driven by space-time white noise.
\newblock {\em Potential Anal.}, 40(1):1--40, 2014.

\bibitem{brehier2}
C.-E. Br{\'e}hier and M.~Kopec.
\newblock {Approximation of the invariant law of SPDEs: error analysis using a
  poisson equation for full-discretization scheme}.
\newblock {\em arXiv: 1311.7030}, 2013.

\bibitem{brehier3}
C.-E. Br{\'e}hier and G.~Vilmart.
\newblock High order integrator for sampling the invariant distribution of a
  class of parabolic stochastic {PDE}s with additive space-time noise.
\newblock {\em SIAM J. Sci. Comput.}, 38(4):A2283--A2306, 2016.

\bibitem{CHW16}
C.~Chen, J.~Hong, and X.~Wang.
\newblock {Approximation of invariant measure for damped stochastic nonlinear
  Schr{\"o}dinger equation via an ergodic numerical scheme}.
\newblock {\em Potential Anal.}, DOI 10.1007/s11118-016-9583-9, 2016.

\bibitem{daprato}
G.~Da~Prato.
\newblock {\em An introduction to infinite-dimensional analysis}.
\newblock Universitext. Springer-Verlag, Berlin, 2006.

\bibitem{BD06}
A.~de~Bouard and A.~Debussche.
\newblock Weak and strong order of convergence of a semidiscrete scheme for the
  stochastic nonlinear {S}chr\"odinger equation.
\newblock {\em Appl. Math. Optim.}, 54(3):369--399, 2006.

\bibitem{DO05}
A.~Debussche and C.~Odasso.
\newblock {Ergodicity for a weakly damped stochastic non-linear Schr\"{o}dinger
  equation}.
\newblock {\em J. Evol. Equ.}, 5:317--356, 2005.

\bibitem{HJK12}
M.~Hutzenthaler, A.~Jentzen, and P.~Kloeden.
\newblock Strong convergence of an explicit numerical method for {SDE}s with
  nonglobally {L}ipschitz continuous coefficients.
\newblock {\em Ann. Appl. Probab.}, 22(4):1611--1641, 2012.

\bibitem{JKW11}
A.~Jentzen, P.~Kloeden, and G.~Winkel.
\newblock Efficient simulation of nonlinear parabolic {SPDE}s with additive
  noise.
\newblock {\em Ann. Appl. Probab.}, 21(3):908--950, 2011.

\bibitem{LS15}
A.~Lang and C.~Schwab.
\newblock Isotropic {G}aussian random fields on the sphere: regularity, fast
  simulation and stochastic partial differential equations.
\newblock {\em Ann. Appl. Probab.}, 25(6):3047--3094, 2015.

\bibitem{L13}
J.~Liu.
\newblock Order of convergence of splitting schemes for both deterministic and
  stochastic nonlinear {S}chr\"odinger equations.
\newblock {\em SIAM J. Numer. Anal.}, 51(4):1911--1932, 2013.

\bibitem{mattingly}
J.~C. Mattingly, A.~M. Stuart, and D.~J. Higham.
\newblock {Ergodicity for SDEs and approximations: locally Lipschitz vector
  fields and degenerate noise}.
\newblock {\em Stochastic Process. Appl.}, 101:185--232, 2002.

\bibitem{MNS13}
B.~E. Moore, L.~Nore{\~n}a, and C.~M. Schober.
\newblock Conformal conservation laws and geometric integration for damped
  {H}amiltonian {PDE}s.
\newblock {\em J. Comput. Phys.}, 232:214--233, 2013.

\bibitem{talay}
D.~Talay.
\newblock {Second order discretization schemes of stochastic differential
  systems for the computation of the invariant law}.
\newblock {\em Rapports de Recherche, Institut National de Recherche en
  Informatique et en Automatique}, 1987.

\bibitem{T02}
D.~Talay.
\newblock Stochastic {H}amiltonian systems: exponential convergence to the
  invariant measure, and discretization by the implicit {E}uler scheme.
\newblock {\em Markov Process. Related Fields}, 8(2):163--198, 2002.
\newblock Inhomogeneous random systems (Cergy-Pontoise, 2001).

\end{thebibliography}
\bibliographystyle{plain}

\end{document}